\documentclass[12pt, a4]{amsart}

\usepackage{latexsym,amsmath,amssymb,amsthm,mathrsfs,color}
\usepackage{float}

\usepackage[bbgreekl]{mathbbol}
\usepackage{bbm}
\usepackage{tikz-cd}
\usepackage[all]{xy}
\usepackage{stmaryrd}
   \usepackage[pagebackref,colorlinks=true, citecolor=brown,
    filecolor=black,
   linkcolor=blue,
   urlcolor=black]{hyperref}

\usepackage[toc,page]{appendix}

\usepackage[shortalphabetic]{amsrefs}
 \makeatletter
\def\append@label@year@{%
    \safe@set\@tempcnta\bib@year
    \edef\bib@citeyear{\the\@tempcnta}%
    \ifnum\bib@citeyear>9
      \append@to@stem{%
          \ifx\bib@year\@empty
          \else
            \@xp\year@short \bib@citeyear \@nil
          \fi
      }%
    \fi
}
\makeatother

\let\oldtocsection=\tocsection
\renewcommand{\tocsection}[2]{\hspace{0em}\oldtocsection{#1}{#2}}

\usepackage[a4paper]{geometry}

\def\upddots{\mathinner{\mkern 1mu\raise 1pt \hbox{.}\mkern 2mu
\mkern 2mu \raise 4pt\hbox{.}\mkern 1mu \raise 7pt\vbox {\kern 7
pt\hbox{.}}} }

\numberwithin{equation}{section}

\begin{document}
\setlength{\unitlength}{2.5cm}

\newtheorem{thm}{Theorem}[section]
\newtheorem{lm}[thm]{Lemma}
\newtheorem{prop}[thm]{Proposition}
\newtheorem{cor}[thm]{Corollary}
\newtheorem{conj}[thm]{Conjecture}
\newtheorem{specu}[thm]{Speculation}

\theoremstyle{definition}
\newtheorem{dfn}[thm]{Definition}
\newtheorem{eg}[thm]{Example}
\newtheorem{rmk}[thm]{Remark}

\newcommand{\N}{\mathbbm{N}}
\newcommand{\R}{\mathbbm{R}}
\newcommand{\C}{\mathbbm{C}}
\newcommand{\Z}{\mathbbm{Z}}
\newcommand{\Q}{\mathbbm{Q}}

\newcommand{\Mp}{{\rm Mp}}
\newcommand{\Sp}{{\rm Sp}}
\newcommand{\GSp}{{\rm GSp}}
\newcommand{\GL}{{\rm GL}}
\newcommand{\PGL}{{\rm PGL}}
\newcommand{\SL}{{\rm SL}}
\newcommand{\SO}{{\rm SO}}
\newcommand{\Spin}{{\rm Spin}}
\newcommand{\GSpin}{{\rm GSpin}}
\newcommand{\Ind}{{\rm Ind}}
\newcommand{\Res}{{\rm Res}}
\newcommand{\Hom}{{\rm Hom}}
\newcommand{\End}{{\rm End}}
\newcommand{\msc}[1]{\mathscr{#1}}
\newcommand{\mfr}[1]{\mathfrak{#1}}
\newcommand{\mca}[1]{\mathcal{#1}}
\newcommand{\mbf}[1]{{\bf #1}}

\newcommand{\mbm}[1]{\mathbbm{#1}}

\newcommand{\into}{\hookrightarrow}
\newcommand{\onto}{\twoheadrightarrow}

\newcommand{\s}{\mathbf{s}}
\newcommand{\cc}{\mathbf{c}}
\newcommand{\bfa}{\mathbf{a}}
\newcommand{\id}{{\rm id}}
\newcommand{\g}{\mathbf{g}_{\psi^{-1}}}
\newcommand{\w}{\mathbbm{w}}
\newcommand{\Ftn}{{\sf Ftn}}
\newcommand{\p}{\mathbf{p}}
\newcommand{\bq}{\mathbf{q}}
\newcommand{\WD}{\text{WD}}
\newcommand{\W}{\text{W}}
\newcommand{\Wh}{{\rm Wh}}
\newcommand{\ggma}{\omega}
\newcommand{\sct}{\text{\rm sc}}
\newcommand{\Of}{\mca{O}^\digamma}
\newcommand{\gk}{c_{\sf gk}}
\newcommand{\Irr}{ {\rm Irr} }
\newcommand{\Irrg}{ {\rm Irr}_{\rm gen} }
\newcommand{\diag}{{\rm diag}}
\newcommand{\uchi}{ \underline{\chi} }
\newcommand{\Tr}{ {\rm Tr} }
\newcommand{\der}\de
\newcommand{\Stab}{{\rm Stab}}
\newcommand{\Ker}{{\rm Ker}}
\newcommand{\bfp}{\mathbf{p}}
\newcommand{\bfq}{\mathbf{q}}
\newcommand{\KP}{{\rm KP}}
\newcommand{\Sav}{{\rm Sav}}
\newcommand{\de}{{\rm der}}
\newcommand{\tnu}{{\tilde{\nu}}}
\newcommand{\lest}{\leqslant}
\newcommand{\gest}{\geqslant}
\newcommand{\tu}{\widetilde}
\newcommand{\tchi}{\tilde{\chi}}
\newcommand{\tomega}{\tilde{\omega}}
\newcommand{\Rep}{{\rm Rep}}
\newcommand{\A}{{\mbf A}}
\newcommand{\BDI}{{\rm Inv}_{\rm BD}}

\newcommand{\cu}[1]{\textsc{\underline{#1}}}
\newcommand{\set}[1]{\left\{#1\right\}}
\newcommand{\ul}[1]{\underline{#1}}
\newcommand{\wt}[1]{\overline{#1}}
\newcommand{\wtsf}[1]{\wt{\sf #1}}
\newcommand{\anga}[1]{{\left\langle #1 \right\rangle}}
\newcommand{\angb}[2]{{\left\langle #1, #2 \right\rangle}}
\newcommand{\wm}[1]{\wt{\mbf{#1}}}
\newcommand{\elt}[1]{\pmb{\big[} #1\pmb{\big]} }
\newcommand{\ceil}[1]{\left\lceil #1 \right\rceil}
\newcommand{\floor}[1]{\left\lfloor #1 \right\rfloor}
\newcommand{\val}[1]{\left| #1 \right|}
\newcommand{\bepsilon}{\overline{\epsilon}}
\newcommand{\HH}{\mca{H}}

\newcommand{\exc}{ {\rm exc} }

\newcommand{\motimes}{\text{\raisebox{0.25ex}{\scalebox{0.8}{$\bigotimes$}}}}
\makeatletter
\newcommand{\extp}{\@ifnextchar^\@extp{\@extp^{\,}}}
\def\@extp^#1{\mathop{\bigwedge\nolimits^{\!#1}}}
\makeatother

\newcommand{\nequiv}{\not \equiv}
\newcommand{\half}{\frac{1}{2}}
\newcommand{\psii}{\widetilde{\psi}}
\newcommand{\ab} {|\!|}
\newcommand{\mb}{{\widetilde{B(\F)}}}

\title[Two-sided cells and splitting Whittaker polynomials]{Two-sided cells of Weyl groups and certain splitting Whittaker polynomials}

\author{Fan Gao and Yannan Qiu}

\address{Fan Gao: School of Mathematical Sciences, Zhejiang University, 866 Yuhangtang Road, Hangzhou, China 310058}
\email{gaofan@zju.edu.cn}
\address{Yannan Qiu: School of Mathematics, Hangzhou Normal University, 2318 Yuhangtang Road, Hangzhou, China 311121}
\email{qiuyannan@hznu.edu.cn}

\maketitle

\begin{abstract} 
Consider the subset of a Weyl group with a fixed special descent set. For Weyl groups of classical types, we determine the number of two-sided cells this subset intersect. Moreover, we apply this result to prove that certain rational Whittaker polynomials associated with covering groups split over the field of rational numbers. 
\end{abstract}

\tableofcontents

\section{Introduction} \label{S:Intro}

Let $W$ be the Weyl group associated with a root system with simple roots 
$$\Delta=\set{\alpha_i: 1\lest i \lest r}.$$
Let $l: W \to \N$ be the length function of $W$. The length $l(w)$ of $w$ is the minimum number of simple reflections $s_{\alpha_i}$, where $\alpha_i \in \Delta$, that appear in a decomposition of $w$. Then one has the Poincar\'e series
$$\mca{P}_W(X):= \sum_{w\in W} X^{l(w)}.$$
Many properties of $\mca{P}_W(X)$ are given for example in \cite[Page 42, Exercise (26)]{BouL2}.

Meanwhile, there is another Poincar\'e series associated with $W$ given as follows. First, let $\Phi$ be the set of roots of the root system underlying $W$. For each $\alpha \in \Phi$, one has the reflection $s_\alpha \in W$. Every $w\in W$ can be written as a product of reflections $\prod_{j=1}^k s_\alpha$, where $s_\alpha$, for $\alpha \in \Phi$, is a reflection that is not necessarily simple. We denote by $l^\sharp(w)$ the minimal number of reflections needed in the decomposition of $w$ as above. This gives a well-defined function 
$$l^\sharp: W \to \N.$$
and thus a modified Poincar\'e series
$$\mca{P}_W^\sharp(X):= \sum_{w\in W} X^{l^\sharp(w)}.$$
Here, the polynomial $\mca{P}_W^\sharp(X)$ and its generalization form one of the foci of our paper.
The series $\mca{P}_W^\sharp(X)$ also satisfies some nice properties, and its study goes back  to at least to the work of Shephard--Todd \cite[\S 5]{ShTo54}.  In particular, it was shown in loc. cit. that
$$\mca{P}_W^\sharp(X) = \prod_{i=1}^r (1 + m_i X),$$
where $m_i, 1\lest i \lest r$ are the exponents of the Weyl group.

The polynomial $\mca{P}_W^\sharp(X)$ appeared in many different and related contexts, for example regarding the hyperplane arrangement, the Euler-Poincar\'e characteristic of the Arnold--Brieskorn manifold, the cohomology of a certain affine Springer fiber; see the work of Sommers \cite{Som97} for an excellent exposition on the connections between these topics. In fact, it was explained in loc. cit. that for ``very good" $n\in \N$ (in the sense of \cite[Definition 3.5]{Som97}) the value of $\mca{P}_W^\sharp(n)$ is dictated by the natural permutation representation $\eta_n$ of $W$ on $\Z[\Delta^\vee]/\Z[n\Delta^\vee]$, where $\Delta^\vee$ denotes the set of simple coroots. More precisely, writing $\chi_{\eta_n}$ for the character of $\eta_n$, one has that for very good $n$,
$$\chi_{\eta_n}(w) = n^{r - l^\sharp(w)}$$
for every $w\in W$. By further exploring this relation, it was shown in \cite[Theorem 3.1]{GGK2} that for very good $n$, the value
$$ \frac{n^r \cdot \mca{P}_W^\sharp(n^{-1})}{\val{W}}= \frac{1}{\val{W}} \cdot \prod_{i=1}^r (n + m_i)$$
is equal to the Whittaker dimension (see \eqref{D:Wh} below) of the  Steinberg representation of the $n$-fold cover $\wt{G}^{(n)}$ of an almost simple simply-connected group $G$, whose root system is of type $\Delta$.
In parallel, the value 
$$\frac{n^r \cdot \mca{P}_W^\sharp(-n^{-1})}{\val{W}}= \frac{1}{\val{W}} \cdot \prod_{i=1}^r (n - m_i)$$
is the Whittaker dimension of the theta representation of $\wt{G}^{(n)}$ for all ``very good" $n$. 

The Steinberg representation and the theta representation mentioned above are both constituents of an unramified genuine principal series $I(\chi_{\rm ex})$ of $\wt{G}^{(n)}$, see \S \ref{SS:rps} for details. The representation $I(\chi_{\rm ex})$ is multiplicity-free and we denote by  ${\rm JH}(I(\chi_{\rm ex}))$ its Jordan--Holder set. Let $\msc{P}(\Delta)$ be the power set of $\Delta$. By a result of Rodier, there is a natural bijection
$$\msc{P}(\Delta) \longrightarrow {\rm JH}(I(\chi_{\rm ex}))$$
denoted by $S \mapsto \pi_S$, where the Steinberg representation and the theta representation are just $\pi_\emptyset$  and $\pi_\Delta$ respectively; see \S \ref{SS:rps}. In view of the above results, we call
$$\mca{P}_{G, \emptyset}(X)= \frac{1}{\val{W}} \cdot \prod_{i=1}^r (X + m_i), \quad \mca{P}_{G, \Delta}(X)= \frac{1}{\val{W}} \cdot \prod_{i=1}^r (X - m_i)$$
the Whittaker polynomial associated with $\emptyset \in \msc{P}(\Delta)$ and $\Delta \in \msc{P}(\Delta)$ respectively. As mentioned, the utility of these two polynomials is that the evaluation $\mca{P}_{G,\emptyset}(n)$ and $\mca{P}_{G,\Delta}(n)$ at ``very good" $n$ gives the values of the Whittaker dimensions.

In general, for every $S \in \msc{P}(\Delta)$, one has a polynomial $\mca{P}_{G, S}(X) \in \Q[X]$ such that $\mca{P}_{G,S}(n)$ is equal to the Whittaker dimension of $\pi_S$ for all very good $n$. It is natural to ask the following question:
\begin{enumerate}
\item[$\bullet$] (Q0) For fixed $G$, determine $S \in \msc{P}(\Delta)$ such that $\mca{P}_{G,S}(X) \in \Q[X]$ is a split polynomial over $\Q$.
\end{enumerate}
Our paper is thus motivated by (Q0) above, and the goal is to give a description of such $S$.

\subsection{Main result}
For every $S \in \msc{P}(\Delta)$, consider the set $\mca{C}_S \subseteq W$ given in \eqref{D:C_S} and denote by $\mfr{C}^{\rm LR}(S)$ the set of two-sided cells of $W$ that intersect $\mca{C}_S$. It turns out that the splitting property of $\mca{P}_{G,S}(X)$ over $\Q$ depends sensitively on the size of $\mfr{C}^{\rm LR}(S)$. Such a relation is implied by the work of Gyoja--Nishiyama--Shimura \cite{GNS99} as we explain in detail in \S \ref{S:WhP}. Roughly speaking, it is desirable to have  $\val{\mfr{C}^{\rm LR}(S)}$ to be as small as possible, in order to get $\mca{P}_{G,S}(X)$ to be splitting.

Motivated by this, in \S \ref{S:cells} we focus on root system $\Delta$ of type $A_r, B_r$ and $D_r$, and we use the Lusztig $\bfa$-function and the Springer correspondence to study the set $\mfr{C}^{\rm LR}(S_j)$, where $S_j:=\set{\alpha_i: \ 1\lest i \lest j} \subseteq \Delta, 0\lest j \lest r$ with the $\alpha_i$'s labelled as by Bourbaki \cite{BouL2}. Here we take $S_0:=\emptyset$ by convention.  The main result of \S \ref{S:cells}  is the following

\begin{thm}[{Theorem \ref{T:ABD}}] \label{T:ABD0} 
Let $\Delta$ be of type $A_r, B_r$ or $D_r$ and let $S_j \in \msc{P}(\Delta), 1\lest j \lest r-1$. Then the following data are determined explicitly:
\begin{enumerate}
\item[--] the size of $\mfr{C}^{\rm LR}(S_j)$;
\item[--] the set $\bfa(\mca{C}_{S_j})$ of values of the $\bfa$-function on $\mca{C}_{S_j}$;
\item[--] the set of special nilpotent orbits associated with the two-sided cells in  $\mfr{C}^{\rm LR}(S_j)$.
\end{enumerate}
More precise results are tabulated in Tables \ref{T:A}, \ref{T:B}, \ref{T:D-odd} and \ref{T:D-even}.
\end{thm}
The proof of Theorem \ref{T:ABD} follows from a case by case analysis. For $\Delta$ of type $A_r$ we also have further result regarding $\mfr{C}^{\rm LR}(T_j)$ for $2\lest j \lest r-1$, where $T_j:=S_j -\set{\alpha_1}$; see Proposition \ref{P:Amore}. Partially relying on these results, we prove in \S\ref{S:splWhP}  the second main result of our paper:

\begin{thm}[{Theorem \ref{T:poly}}] \label{T:poly0}
Let $G$ be simply-connected and almost simple. For every $S \in \msc{P}(\Delta)_\flat \cup \msc{P}(\Delta)_\flat^*$ defined in \eqref{D:Pflat} and \eqref{D:Pflat*}, the Whittaker polynomial $\mca{P}_{G,S}(X) \in \Q[X]$ splits over $\Q$.
\end{thm}

In fact, in view of the definition of $\msc{P}(\Delta)_\flat$ and $\msc{P}(\Delta)_\flat^*$, it suffices to consider $\mca{P}_{G,S}(X)$ for $S \in \msc{P}(\Delta)_\flat$. Moreover, the only non-trivial cases that need to be analyzed are types  $A_r, B_r, C_r, D_r$ and $G_2$. For type $A_r, B_r$ and $C_r$, we apply Theorem \ref{T:ABD0}, Proposition \ref{P:Amore} and the result of Gyoja--Nishiyama--Shimura \cite{GNS99} mentioned above. The case of $D_r$ follows from a more direct computation, while  the result for type $G_2$ can be directly extracted from \cite{Ga6}.  The discussion in each case, whenever $\mca{P}_{G,S}(X)$ splits, also gives more precise information on the roots of $\mca{P}_{G,S}(X)$.

We expect that the converse of Theorem \ref{T:poly0} also holds, and thus $\mca{P}_{G,S}(X)$ splits if and only if $S \in \msc{P}(\Delta)_\flat \cup \msc{P}(\Delta)_\flat^*$, see the brief discussion in \S \ref{SS:spec}. We also need to add that, as pointed out by the referee, there are very intriguing connections between these Whittaker polynomials $\mca{P}_{G,S}(X)$ and the Kirkman numbers and polynomials as discussed in the work of Armstrong--Reiner--Rhoades, see especially \cite[\S 9]{ARR15}. Also, it would be interesting to exploit further connection between the splitting properties of $\mca{P}_{G,S}(X)$ and certain exponents of the restricted hyperplane arrangements, 

Through out the paper, we have $\N:= \Z_{\gest 0}$ and for every $x, y\in \R$, we write
$$[x, y]_\N:=[x, y] \cap \N.$$
Also, $\floor{x} \in \Z$ denotes the integral part of $x\in \R$.
\subsection{Acknowledgement} 

We would like to thank Nanhua Xi, Zhanqiang Bai, Xun Xie for very helpful discussions on the topics of cells and the $\bfa$-function. Thanks are also due to the referee for some very insightful comments and suggestions.

The work of Y. Q. is partially supported by NSFC-12171030 and LQ24A010010.
The work of F. G. is partially supported by the National Key R\&D Program of China (No. 2022YFA1005300) and also by NSFC-12171422.

\section{Two-sided cells intersecting $\mca{C}_S$} \label{S:cells}

\subsection{Cells of $W$}
Let $W$ be the Weyl group of an irreducible reduced root system with simple roots 
$$\Delta=\set{\alpha_i: \ 1\lest i \lest r}.$$ Thus, $W$ is generated by the simple reflections $s_\alpha, \alpha \in \Delta$. 
Denote by $\Irr(W)$ the set of isomorphism classes of irreducible representations of $W$.
We write $\mbm{1}_W$ and $\varepsilon_W$ for the trivial and sign characters of $W$ respectively.
For a finite-dimensional representation $\sigma$ of $W$, we write ${\rm JH}(\sigma)$ for the Jordan--Holder set of irreducible constituents of $\sigma$, counted without multiplicities.

Let $l: W \to \N$ be the usual length function of $W$ defined by $\Delta$. For every $w$, we have the left and right descent set of $w$ given as follows
$${\rm Desc}_{\rm L}(w):=\set{\alpha \in \Delta: \ l(s_\alpha w) < l(w)}, \quad {\rm Desc}_{\rm R}(w):=\set{\alpha \in \Delta: \ l(w s_\alpha) < l(w)}.$$ 
Using the Bruhat order, the Kazhdan--Lusztig polynomial and the descent set function ${\rm Desc}_{\rm L}$ (resp. ${\rm Desc}_{\rm R}$), one has preorders  $\lest_{\rm L}$ and $\lest_{\rm R}$ defined on $W$, see \cite{KL1}. Define an equivalence relation
$x \sim_{\rm L} y$
by requiring $x \lest_{\rm L} y$ and $y \lest_{\rm L} x$; similarly we have $x \sim_{\rm R} y$ using $\lest_{\rm R}$.
 The resulting equivalence classes are called the left and right cells of $W$ respectively. We have 
$$x \sim_{\rm L} y  \text{ if and only if } x^{-1} \sim_{\rm R} y^{-1}.$$
One can further combine the left and right equivalence above and write
$$x \sim_{\rm LR} y$$
if there exists $z\in W$ such that $x \sim_{\rm L} z\sim_{\rm R} y$ holds (see \cite[Page 137 and Corollary 12.16]{Lus84-B}). The resulting equivalence classes are called two-sided cells of $W$. For $\heartsuit \in \set{{\rm left}, {\rm right}}$, we set 
$$\mfr{C}^\heartsuit:=\set{\mca{C} \subseteq W: \ \mca{C} \text{ is a $\heartsuit$-cell in } W}.$$
and also
$$\mfr{C}^{\rm LR}:=\set{\mca{C} \subseteq W: \ \mca{C} \text{ is a  two-sided cell in } W}.$$
Every two-sided cell $\mca{C} \in \mfr{C}^{\rm LR}$ is a disjoint union of left cells, also of right cells. 

Associated with every $\heartsuit$-cell $\mca{C} \subseteq W$ is a Weyl group representation $\sigma_\mca{C}$, which may not be irreducible in general. Naturally, for every $\mca{C} \in \mfr{C}^{\rm LR}$ one has a Weyl group representation $\sigma_\mca{C}$. Moreover, we get
$$\C[W] =\bigoplus_{\mca{C} \in \mfr{C}^\heartsuit} \sigma_\mca{C}$$
and also
$$\C[W] = \bigoplus_{\mca{C} \in \mfr{C}^{\rm LR}} \sigma_\mca{C}.$$
Here if $\mca{C}, \mca{C}' \in \mfr{C}^{\rm LR}$ are distinct, then ${\rm JH}(\sigma_\mca{C}) \cap {\rm JH}(\sigma_{\mca{C}'}) =\emptyset$, i.e., $\sigma_\mca{C}$ and $\sigma_{\mca{C}'}$ have no isomorphic irreducible constituents in common, see \cite{Lus82}. 

\begin{dfn} \label{D:LR}
Two elements $\sigma,  \sigma' \in \Irr(W)$ are called in the same family or in the same two-sided cell, which we denote by $\sigma \sim_{\rm LR} \sigma'$, if $\sigma, \sigma' \in {\rm JH}(\sigma_\mca{C})$ for a $\mca{C} \in \mfr{C}^{\rm LR}$.
\end{dfn}

Let 
$$\Irr(W)^{\rm spe} \subset \Irr(W)$$
be the subset of special representations given by Lusztig \cite{Lus79}. For every two-sided cell $\mca{C} \in \mfr{C}^{\rm LR}$, the representation $\sigma_\mca{C}$ contains a unique special representation $\rho_\mca{C}^{\rm spe} \in \Irr(W)^{\rm spe}$, thus necessarily with multiplicity $\dim \rho_\mca{C}^{\rm spe}$. This gives a bijection
\begin{equation} \label{E:cel-spe}
\mfr{C}^{\rm LR} \longrightarrow \Irr(W)^{\rm spe}, \quad \mca{C} \mapsto \rho_\mca{C}^{\rm spe}.
\end{equation}
See \cite[Theorem 5.25]{Lus84-B}, \cite{BV4, BV5} or \cite{Gec12} for detailed discussions on this.

\subsection{The representation $\sigma_S$} \label{SS:sigS}
For every subset $S \subseteq \Delta$, consider the set
\begin{equation} \label{D:C_S}
\mca{C}_S:=\set{w\in W: {\rm Desc}_{\rm L}(w) = S} \subseteq W.
\end{equation}
Since the function ${\rm Desc}_{\rm L}(\cdot)$ is constant on right cells (see \cite[Proposition 2.4]{KL1}), it follows that 
$$\mca{C}_S = \bigsqcup_{i\in I} \mca{C}_i,$$
where $\mca{C}_i$ is a  right cell in $W$; that is, $\mca{C}_S$ is a disjoint union of right cells of $W$. Every $\mca{C}_i$ gives rise to a right cell representation  $\sigma_{\mca{C}_i}$ of $W$. From this, we define
$$\sigma_S:=\bigoplus_{i \in I} \sigma_{\mca{C}_i},$$
and call it the right cell representation of the Weyl group associated with $S$.

The main question we consider in this section is the following:
\begin{enumerate}
\item[$\bullet$] (Q1) How many two-sided cells of $W$ does $\mca{C}_S$ intersect?
\end{enumerate}
Denote
$$\mfr{C}^{\rm LR}(S):=\set{\mca{C} \in \mfr{C}^{\rm LR}: \ \mca{C} \cap \mca{C}_S \ne \emptyset} \subseteq \mfr{C}^{\rm LR}.$$
Then, (Q1) is equivalent to computing the size of $\mfr{C}^{\rm LR}(S)$. This question is of interest on its own, but is also motivated from the problem of determining whether a certain Whittaker polynomial in $\Q[X]$ splits over $\Q$ or not, especially for groups of classical types, see \S \ref{S:splWhP} . In fact, for our purpose, we are interested in the weaker question:
\begin{enumerate}
\item[$\bullet$] (Q1w) For which subset $S \subseteq \Delta$ one has $\val{\mfr{C}^{\rm LR}(S)} \lest 2$?
\end{enumerate}

Through out the paper, we use $\msc{P}(\Delta)$ to denote the power set of $\Delta$. For every $S  \in \msc{P}(\Delta)$ we write $W(S) \subseteq W$ for the parabolic Weyl subgroup  generated by $\set{s_\alpha: \alpha\in S}$. Denote by
$$w_S \in W(S) \subseteq W$$
the unique longest Weyl element in $W(S)$. In particular, $w_\Delta \in W$ is the longest Weyl element in $W$.
We also write
$$S^*:=\Delta -S$$
for every $S \in \msc{P}(\Delta)$. Given any subset $S \subseteq \Delta$, we have two special elements $w_S, w_{S^*} \cdot w_\Delta \in  \mca{C}_S$. In particular, 
$$\mca{C}_\emptyset = \set{1}, \mca{C}_\Delta = \set{w_\Delta}$$
and $\sigma_\emptyset = \mbm{1}_W, \sigma_\Delta = \varepsilon_W$. Also $\mfr{C}^{\rm LR}(\emptyset)$ and $\mfr{C}^{\rm LR}(\Delta)$ are both singleton sets, containing $\set{1}$ and $\set{w_\Delta}$ respectively.

\begin{prop} \label{P:dual}
Let $S \in \msc{P}(\Delta)$ be arbitrary. For $x\in W$ one has that $x\in \mca{C}_S$ if and only if
\begin{equation} \label{E:bds}
 w_{S^*} \cdot w_\Delta \lest_{\rm R} x \lest_{\rm R}  w_{S}.
\end{equation}
Thus, $\mca{C}_{S^*} =\mca{C}_S \cdot w_\Delta$. Moreover, we have $\sigma_{S^*} \simeq \varepsilon_W \otimes \sigma_S$ and $\val{ \mfr{C}^{\rm LR}(S^*) } = \val{ \mfr{C}^{\rm LR}(S) }$.
\end{prop}
\begin{proof} 
Since $x \lest_{\rm R} y$ implies ${\rm Desc}_{\rm L}(x) \supseteq {\rm Desc}_{\rm L}(y)$, the if part of the first assertion is clear. Now, for every $x\in \mca{C}_S$ we have ${\rm Desc}_{\rm L}(x) = S$ and also
$$l(w_{S^*} \cdot x) = l(w_{S^*}) + l(x),$$
see \cite[Lemma 9.7]{Lus03-B}. We have $w_\Delta = w_{S^*}\cdot  x \cdot u$ for $u \in W$ satisfying 
$$l(w_\Delta) = l(w_{S^*}) + l(x) + l(u).$$
So $w_{S^*}\cdot w_\Delta = x \cdot u$ and by the definition of $\lest_{\rm R}$, we have $w_{S^*}\cdot w_\Delta \lest_{\rm R} x$.
On the other hand, let $v\in W$ be such that $x= w_S \cdot v$ and $l(x) = l(w_S) + l(v)$. Similarly, $x \lest_{\rm R} w_S$. Thus, the first chain of preorders is proved. The map $x \mapsto x \cdot w_\Delta$ on $W$  reverses the preorder $\lest_{\rm R}$ on $W$, and thus induces an involution on the left, right or two-sided cells on $W$, see \cite[Remark 3.3. a)]{KL1}. It follows that $\mca{C}_{S^*} = \mca{C}_S \cdot w_\Delta$.

This equality coupled with properties of right cell representations of $W$ give the isomorphism $\sigma_{S^*} \simeq \varepsilon_W \otimes \sigma_S$, see \cite[Proposition 6.3.5]{BB05-B}. Lastly, as mentioned above, the map $x \mapsto x \cdot w_\Delta$ induces a bijection on $\mfr{C}^{\rm LR}$. It clearly gives rise to a bijection between $\mfr{C}^{\rm LR}(S^*)$ and $\mfr{C}^{\rm LR}(S)$, whence the last assertion.
\end{proof}

To tackle (Q1) or (Q1w), one can consider the explicit decomposition of $\sigma_S$ into irreducible representations of $W$:
$$\sigma_S = \bigoplus_{\rho \in \Irr(W)} \mfr{m}(\rho, \sigma_S) \cdot \rho,$$
where $\mfr{m}(\rho, \sigma_S) \in \N$ denotes the multiplicity. For every $\rho \in \Irr(W)$ occurring in $\sigma_S$, there is a unique $\rho^\sharp \in \Irr(W)^{\rm spe}$ such that $\rho \sim_{\rm LR} \rho^\sharp$.  Here $\rho^\sharp$ is equal to $\rho_\mca{C}^{\rm spe}$ in the notation of \eqref{E:cel-spe} for a unique $\mca{C} \in \mfr{C}^{\rm LR}(S)$.
Define
\begin{equation} \label{E:IrrS}
\Irr(W)^{\rm spe}_S:=\set{\rho^\sharp: \  \rho \in \Irr(W) \text{ and } \mfr{m}(\rho, \sigma_S) \ne 0}.
\end{equation}
By restriction, the map $\mca{C} \mapsto \rho_\mca{C}^{\rm spe}$ in \eqref{E:cel-spe} gives a well-defined natural bijection
\begin{equation} \label{E:bij-s}
\mfr{C}^{\rm LR}(S) \longrightarrow \Irr(W)^{\rm spe}_S.
\end{equation}
Hence, granted with a knowledge of the multiplicity set $\set{\mfr{m}(\rho, \sigma_S)}_{\rho\in \Irr(W)}$, one can determine $\Irr(W)_{\rm spe}^S$ and thus also the set $\mfr{C}^{\rm LR}(S)$.

\subsection{More explicit formulas for $\sigma_S$}
One has a more explicit formula of $\sigma_S$ given as follows (see \cite[Theorem 6.3.8]{BB05-B})
\begin{equation} \label{E:exf1}
\sigma_S \simeq \bigoplus_{Z:\ S \subseteq Z \subseteq \Delta} (-1)^{\val{Z-S}} \cdot \Ind_{W(Z)}^W \varepsilon_{W(Z)}.
\end{equation}
Coupled with the isomorphism
$$\sigma_{S^*} \simeq \varepsilon_W \otimes \sigma_S$$
from Proposition \ref{P:dual}, we get that
\begin{equation} \label{E:exf2}
\sigma_S \simeq \varepsilon_W \otimes \sigma_{S^*} \simeq \bigoplus_{Z: \ Z \subseteq S} (-1)^{\val{S-Z}} \cdot \Ind_{W(Z^*)}^W \mbm{1}_{W(Z^*)}.
\end{equation}

The decomposition of $\Ind_{W(S)}^W \mbm{1}_{W(S)}$ into the irreducible representations of $W$ follows from (but is weaker than) a computation of the Green's polynomial involving the Lusztig--Shoji algorithm (\cite{Spr76, Lus-CSV, Sho88, Kim18}). Nevertheless, it is still a nontrivial problem. However, if $S \in \msc{P}(\Delta)$ is ``close" to $\emptyset$ or $\Delta$, then a direct computation of the decomposition of $\sigma_S$ using \eqref{E:exf1} or \eqref{E:exf2}  is amenable; this will in turn determine $\mfr{C}^{\rm LR}(S)$ efficiently.

\subsection{$\bf{a}$-function and upper bound of $\val{\mfr{C}^{\rm LR}(S)}$}
Another method we use to compute $\mfr{C}^{\rm LR}(S)$ is to use the $\bfa$-function
$$\bfa: W \longrightarrow \N_{\gest 0}$$
introduced by Lusztig \cite[\S 2]{Lus-CAW1}. It is known that the $\bfa$-function is constant on two-sided cells in $W$, see \cite[Theorem 5.4]{Lus-CAW1}. As examples, $\bfa(1)=0$, and $\bfa(w_\Delta)$ is equal to the number of positive roots.

\begin{lm}\label{L:YN12}
Let $S \in \msc{P}(\Delta)$ be arbitrary.
\begin{enumerate}
\item[(i)] One has $\bfa(w_{S^*}\cdot w_\Delta) \gest \bfa(x) \gest \bfa(w_{S})$
for every $x \in \mca{C}_S$.
\item[(ii)] If $\val{\bfa(w_{S^*} \cdot w_\Delta) - \bfa(w_{S})} \lest 1$, then 
$x \sim_{\rm R} w_{S^*} \cdot w_\Delta$ or $x \sim_{\rm R} w_{S}$; in particular, we have $\val{\mfr{C}^{\rm LR}(S)} \lest 2$ in this case.
\item[(iii)] If $\bfa(w_{S^*} \cdot w_\Delta) = \bfa(w_{S})$, then $\val{\mfr{C}^{\rm LR}(S)}=1$.
\end{enumerate}
\end{lm}
\begin{proof}
First, (i) follows from \eqref{E:bds} and \cite[Theorem 5.4]{Lus-CAW1}. Now for (ii), the assumption together with (i) shows that $\bfa(x) = \bfa(w_{S^*} \cdot w_\Delta)$ or $\bfa(x)=\bfa(w_{S})$ for every $x\in \mca{C}_S$. We also have $w_{S^*} \cdot w_\Delta \lest_{\rm R} x \lest_{\rm R} w_{S}$ from Proposition \ref{P:dual} . It then follows from \cite[Corollary 1.9]{Lus-CAW2} that $x \sim_{\rm R} w_{S^*} \cdot w_\Delta$ or $x \sim_{\rm R} w_{S}$. The rest of (ii) is clear. The proof of (iii) is completely analogous to that of (ii) and we omit the details.
\end{proof}

The set $\mfr{C}^{\rm LR}$ is endowed with a partial order given as follows. For $\mca{C}, \mca{C}' \in \mfr{C}^{\rm LR}$ we write $\mca{C} \lest_{\rm LR} \mca{C}'$ if there is $x \in \mca{C}, y\in \mca{C}'$ such that $x\lest_{\rm LR} y$. Let $\msc{N}^{\rm spe}$ denote the set of special nilpotent orbits of the complex Lie algebra $\mfr{g}_\C$ with underlying root system $\Delta$. In fact, by definition the set $\msc{N}^{\rm spe}$ corresponds to the special irreducible Weyl group representations via the Springer correspondence
$${\rm Spr}: \Irr(W)^{\rm spe} \longrightarrow \msc{N}^{\rm spe}.$$
Also, $\msc{N}^{\rm spe}$ is partially ordered by topological closure, i.e., $\mca{O}_1 \lest \mca{O}_2$ if $\mca{O}_1 \subseteq \overline{\mca{O}_2}$. The composite of ${\rm Spr}$ with the map in \eqref{E:cel-spe} gives a bijection
$$f_{\rm co}: \mfr{C}^{\rm LR} \longrightarrow \Irr(W)^{\rm spe} \longrightarrow \msc{N}^{\rm spe}.$$
The map $f_{\rm co}$ preserves the partial order on $\mfr{C}^{\rm LR}$ and that on $\msc{N}^{\rm spe}$, see \cite[Corollary 5.6]{Gec12} or \cite[Theorem 4, b)]{Bez09}.

\begin{lm} \label{L:co}
For every $\Delta$ and $S \subseteq \Delta$, let $\mca{C}_{\rm max}^S \in \mfr{C}^{\rm LR}(S)$ (resp.  $\mca{C}_{\rm min}^S \in \mfr{C}^{\rm LR}(S)$) be the two-sided cell containing $w_S$ (resp. $w_{S^*} w_\Delta$).
\begin{enumerate}
\item[(i)] We have $\mca{C}_{\rm min}^S \lest_{\rm LR} \mca{C}_{\rm max}^S$ and also
$$\mca{C}_{\rm min}^S \lest_{\rm LR} \mca{C} \lest_{\rm LR} \mca{C}_{\rm max}^S$$
for every $\mca{C} \in \mfr{C}^{\rm LR}(S)$ (and thus the notation) .
\item[(ii)] By restriction, one has an injection 
$$\begin{tikzcd}
f_{\rm co}: \mfr{C}^{\rm LR}(S) \ar[r, hook]  & \msc{N}^{\rm spe} \cap [f_{\rm co}(\mca{C}_{\rm min}^S), f_{\rm co}(\mca{C}_{\rm max}^S)],
\end{tikzcd}$$
where the codomain denotes the set of special nilpotent orbits lying between $f_{\rm co}(\mca{C}_{\rm min}^S)$ and $f_{\rm co}(\mca{C}_{\rm max}^S)$.
\end{enumerate}
\end{lm}
\begin{proof}
For (i), we have 
$${\rm Desc}_{\rm L}(w_S) = S = {\rm Desc}_{\rm L}(w_{S^*} w_\Delta)$$
 and $w_{S^*} w_\Delta = w_S \cdot (w_S w_{S^*} w_\Delta)$. 
 This implies that $w_{S^*} w_\Delta \lest_{\rm R}w_S $. For every $x\in \mca{C}_S$ one has $w_{S^*} w_\Delta \lest_{\rm R} x \lest_{\rm R} w_S$, and thus $\mca{C}_{\rm min}^S \lest_{\rm LR} \mca{C} \lest_{\rm LR} \mca{C}_{\rm max}^S$ for every $\mca{C} \in \mfr{C}^{\rm LR}(S)$. Also, (ii) follows immediately from (i).
\end{proof}

We conveniently write 
$$\mca{O}_{\rm min}^S:=f_{\rm co}(\mca{C}_{\rm min}^S), \quad \mca{O}_{\rm max}^S:=f_{\rm co}(\mca{C}_{\rm max}^S)$$ and also set
$$\msc{N}^{\rm spe}(S):=\msc{N}^{\rm spe} \cap [\mca{O}_{\rm min}^S, \mca{O}_{\rm max}^S].$$

For each Cartan type, we write $\Delta:=\set{\alpha_1, \alpha_2, ..., \alpha_r}$ using the Bourbaki labelling as in \cite[Page 265--290]{BouL2}. For every $0\lest j \lest r$ we define
$$S_j:= \set{\alpha_1, \alpha_2, ..., \alpha_j} \subseteq \Delta,$$
where for convention we take 
$$S_0:=\emptyset.$$
Since the Weyl groups of type $B_r$ and $C_r$ are the same, we only  discuss the case when $\Delta$ is of type $B_r$ here. The reason why we consider only $S_j$ as above (instead of general $S \subseteq \Delta$) is motivated from our consideration in \S \ref{S:WhP}--\S \ref{S:splWhP} regarding the splitting properties of certain Whittaker polynomials. In fact, we only consider $\Delta$ of classical type and their associated $S_j$ as above. We mention in passing also that for type $A_r$ groups, one also need to consider subset of $\Delta$ of the form $S_j - \set{\alpha_1}$, for the same purpose of Whittaker polynomials; this we discuss in \S \ref{SS:Amore}.

Thus, for $\Delta$ of classical type $\diamondsuit \in \set{A_r, B_r, D_r}$, we define a function 
$$\varphi_\diamondsuit:  [0, r]_\N  \longrightarrow \N$$
given by
$$\varphi_\diamondsuit(j):=\val{\mfr{C}^{\rm LR}(S_j)}.$$
One has $\varphi_\diamondsuit(0) = \varphi_\diamondsuit(r)=1$. For convenience of later reference, we reproduce from \cite{BouL2} the Dynkin diagram for $\Delta$ of type $A_r, B_r, D_r$ and $G_2$ below:

\begin{table}[H] 
\caption{Dynkin diagrams for $A_r, B_r, D_r$ and $G_2$}
 \label{T:ABDG}
$$\qquad 
\begin{picture}(4.7,0.2)(0,0)
\put(1,0){\circle{0.08}}
\put(1.5,0){\circle{0.08}}
\put(2,0){\circle{0.08}}
\put(2.5,0){\circle{0.08}}
\put(3,0){\circle{0.08}}
\put(1.04,0){\line(1,0){0.42}}
\multiput(1.55,0)(0.05,0){9}{\circle{0.02}}
\put(2.04,0){\line(1,0){0.42}}
\put(2.54,0){\line(1,0){0.42}}
\put(1,0.1){\footnotesize $\alpha_{1}$}
\put(1.5,0.1){\footnotesize $\alpha_{2}$}
\put(2,0.1){\footnotesize $\alpha_{r-2}$}
\put(2.5,0.1){\footnotesize $\alpha_{r-1}$}
\put(3,0.1){\footnotesize $\alpha_{r}$}
\end{picture}
$$
\vskip 10pt

$$ \qquad 
\begin{picture}(4.7,0.2)(0,0)
\put(1,0){\circle{0.08}}
\put(1.5,0){\circle{0.08}}
\put(2,0){\circle{0.08}}
\put(2.5,0){\circle{0.08}}
\put(3,0){\circle{0.08}}
\put(1.04,0){\line(1,0){0.42}}
\multiput(1.55,0)(0.05,0){9}{\circle*{0.02}}
\put(2.04,0){\line(1,0){0.42}}
\put(2.54,0.015){\line(1,0){0.42}}
\put(2.54,-0.015){\line(1,0){0.42}}
\put(2.68,-0.05){{\large $>$}}
\put(1,0.1){\footnotesize $\alpha_1$}
\put(1.5,0.1){\footnotesize $\alpha_2$}
\put(2,0.1){\footnotesize $\alpha_{r-2}$}
\put(2.5,0.1){\footnotesize $\alpha_{r-1}$}
\put(3,0.1){\footnotesize $\alpha_r$}
\end{picture}
$$
\vskip 10pt

$$
\begin{picture}(4.7,0.4)(0,0)
\put(1,0){\circle{0.08}}
\put(1.5,0){\circle{0.08}}
\put(2,0){\circle{0.08}}
\put(2.5,0){\circle{0.08}}
\put(3,0){\circle{0.08}}
\put(3.5, 0.25){\circle{0.08}}
\put(3.5, -0.25){\circle{0.08}}
\put(1.04,0){\line(1,0){0.42}}
\put(1.54,0){\line(1,0){0.42}}
\multiput(2.05,0)(0.05,0){9}{\circle{0.02}}
\put(2.54,0){\line(1,0){0.42}}
\put(3.03,0.03){\line(2,1){0.43}}
\put(3.03,-0.03){\line(2,-1){0.43}}
\put(1,0.1){\footnotesize $\alpha_1$}
\put(1.5,0.1){\footnotesize $\alpha_2$}
\put(2,0.1){\footnotesize $\alpha_3$}
\put(2.5,0.1){\footnotesize $\alpha_{r-3}$}
\put(2.9,0.15){\footnotesize $\alpha_{r-2}$}
\put(3.5,0.35){\footnotesize $\alpha_{r-1}$}
\put(3.5,-0.4){\footnotesize $\alpha_r$}
\end{picture}
$$
\vskip 10pt

$$
\begin{picture}(5.7,0.2)(0,0)
\put(2.5,0){\circle{0.08}}
\put(3,0){\circle{0.08}}
\put(2.53,0.018){\line(1,0){0.44}}
\put(2.54,0){\line(1,0){0.42}}
\put(2.53,-0.018){\line(1,0){0.44}}
\put(2.68,-0.05){{\large $<$}}
\put(2.5,0.1){\footnotesize $\alpha_1$}
\put(3,0.1){\footnotesize $\alpha_2$}
\end{picture}
$$
\end{table}

We will compute for $\diamondsuit \in \set{A_r, B_r, D_r}$ and every $j \in [0, r]_\N$ the function $\varphi_\diamondsuit(j)$. For fixed $S_j, 0\lest j \lest r$, we will study $\mfr{C}^{\rm LR}(S_j^*)$ and the strategy we implement for this is the following:
\begin{enumerate}
\item[(S1)] We first compute the values $\bfa(w_{S_j^*})$ and $\bfa(w_{S_j} \cdot w_\Delta)$, which then gives the integral interval $[\bfa(w_{S_j^*}), \bfa(w_{S_j} \cdot w_\Delta)]_\N$ with length 
$$\mfr{m}_j:=\bfa(w_{S_j} \cdot w_\Delta) - \bfa(w_{S_j^*})  +1.$$
It follows from Lemma \ref{L:YN12} (i) that $\bfa(x) \in [\bfa(w_{S_j^*}), \bfa(w_{S_j} \cdot w_\Delta)]_\N$ for every $x\in \mca{C}_{S_j^*}$.
\item[(S2)] For every $v\in [\bfa(w_{S_j^*}), \bfa(w_{S_j} \cdot w_\Delta)]_\N$ we give an element $x\in \mca{C}_{S_j^*}$ such that $\bfa(x)=v$. That is, the function 
$$\bfa: \mca{C}_{S_j^*} \longrightarrow [\bfa(w_{S_j^*}), \bfa(w_{S_j} \cdot w_\Delta)]_\N$$
is surjective. In particular, $\varphi_\diamondsuit(j) \gest \mfr{m}_j$.
\item[(S3)] We compute $\mca{O}_{\rm min}^{S_j^*}$ and $\mca{O}_{\rm max}^{S_j^*}$, and thus determine $\msc{N}^{\rm spe}(S_j^*)$. We show that $\val{ \msc{N}^{\rm spe}(S_j^*) } = \mfr{m}_j$. It then follows from Lemma \ref{L:co} (iii) that $\varphi_\diamondsuit(j) \lest \mfr{m}_j$. One concludes $\varphi_\diamondsuit(j) = \mfr{m}_j$.
\end{enumerate}

Before we state our result, we introduce some notation. First, for every $1\lest i \lest r$, we write
$s_i:= s_{\alpha_i} \in W$ for the simple reflection associated with $\alpha_i$. For $\Delta$ of type $B_r$, for every $1\lest i \lest r$ and $1\lest q \lest r-1$ we write
\begin{equation} \label{D:bb}
\mbm{b}_{i,q}:=s_i s_{i+1} ... s_{r-1} s_r s_{r-1} ... s_{q+1} s_q \in W(B_r).
\end{equation}
Meanwhile, for $\Delta$ of type $D_r$ and every $1\lest i \lest r$ and $1\lest q \lest r-2$, we write
\begin{equation} \label{D:dd}
\mbm{d}_{i,q}:=s_i s_{i+1} ... s_{r-1} s_r s_{r-2} s_{r-3} ... s_{q+1} s_q \in W(D_r),
\end{equation}
and set $\mbm{d}_{r,r-1}:=\mbm{d}_{r,r}:=s_r \in W(D_r)$. Also, let $\tau\in {\rm Aut}(W)$ be the unique element such that 
$$\tau(s_r) = s_{r-1}, \quad \tau(s_l)=s_l$$
for every $1\lest l\lest r-2$. 

The main result in this section is as follows:

\begin{thm} \label{T:ABD}
Let $\Delta$ be of type $\diamondsuit \in \set{A_r, B_r, D_r}$, and let $S_j \in \msc{P}(\Delta), 1 \lest j \lest r-1$. Then the assertions in (S1)--(S3) hold with the following precise data tabulated in Tables \ref{T:A}, \ref{T:B}, \ref{T:D-odd} and \ref{T:D-even} for type $A_r, B_r$, $D_r$ with $r$ odd, and $D_r$ with $r$ even, respectively:
\begin{enumerate}
\item[--] the value of $\varphi_\diamondsuit(j)=\mfr{m}_j$;
\item[--]  the set $\bfa(\mca{C}_{S_j^*})$, or equivalently, the interval $[\bfa(w_{S_j^*}), \bfa(w_{S_j} \cdot w_\Delta)]_\N$ of length $\mfr{m}_j$;
\item[--]  a set $X_j:=\set{x_k: 1\lest k \lest \mfr{m}_j} \subset \mca{C}_{S_j^*}$ such that $\bfa(X_j)=\bfa(\mca{C}_{S_j^*})\in[\bfa(w_{S_j^*}), \bfa(w_{S_j} \cdot w_\Delta)]_\N$;
\item[--]  the set $\msc{N}^{\rm spe}(S_j^*)$.
\end{enumerate}
In particular, we see that the $\bfa$-function separates two-sided cells on $\mfr{C}^{\rm LR}(S_j^*)$ for every $1\lest j \lest r-1$; also, $\varphi_\diamondsuit(-)$ is a monotone non-decreasing function on $[0, r-1]_\N$.
\end{thm}

\begin{table}[H] 
\caption{Type $A_r$ and $S_j, 1\lest j \lest r-1$}
 \label{T:A}
\vskip 10pt
\renewcommand{\arraystretch}{1.4}
\begin{tabular}{|c|c|c|c|c|c|c|c|c|c|c|}
\hline
$\varphi_{A_r}(j)$ & $\bfa(\mca{C}_{S_j^*})$  &  $X_j$ & $\msc{N}^{\rm spe}(S_j^*)$    \\
\hline
\hline
$1$ & $\frac{(r-j)(r-j+1)}{2}$ & $ w_{S_j^*}$,  & $\mca{O}_{(j+1, 1^{r-j})}$   \\ 
\hline
\end{tabular}
\end{table}
\vskip 10pt

\begin{table}[H] 
\caption{Type $B_r$ and $S_j, 1\lest j \lest r-1$}
 \label{T:B}
\vskip 10pt
\renewcommand{\arraystretch}{1.4}
\begin{tabular}{|c|c|c|c|c|c|c|c|c|c|c|}
\hline
$\varphi_{B_r}(j)$ & $\bfa(\mca{C}_{S_j^*})$  &  $X_j$ & $\msc{N}^{\rm spe}(S_j^*)$    \\
\hline
\hline
$\floor{\frac{j}{2}}+1$ & $(r-j)^2+k$, & $(\prod_{a=-k}^{-1} \mbm{b}_{j+1, j+1+2a}) \cdot w_{S_j^*}$,  & $\mca{O}_{(2j+1-2k, 2k+1, 1^{2r-2j-1})}$,   \\ 
& $0\lest k \lest \floor{\frac{j}{2}}$ & $0\lest k \lest \floor{\frac{j}{2}}$ & $0\lest k \lest \floor{\frac{j}{2}}$\\
\hline
\end{tabular}
\end{table}
\vskip 10pt

\begin{table}[H]  
\caption{Type $D_r, r$ odd  and $S_j, 1\lest j \lest r-1$}
\label{T:D-odd}
\vskip 10pt
\renewcommand{\arraystretch}{1.4}
\begin{tabular}{|c|c|c|c|c|c|c|c|c|c|c|}
\hline
$\varphi_{D_r}(j)$ & $\bfa(\mca{C}_{S_j^*})$  &  $X_j$ & $\msc{N}^{\rm spe}(S_j^*)$    \\
\hline
\hline
$\floor{\frac{j+1}{2}}+1$, & $(r-j)(r-j-1)+k$, & $(\prod_{a=-k}^{-1} \mbm{d}_{j+1, j+1+2a} )\cdot w_{S_j^*}$,  & $\big\{ \mca{O}_{(2j+1-2k, 2k+1, 1^{2r-2j-2})}:$   \\ 
if $0\lest j \lest r-2$ & $0\lest k \lest \floor{\frac{j+1}{2}}$ & $0\lest k \lest \floor{\frac{j+1}{2}}$ & $0\lest k \lest \floor{\frac{j}{2}} \big\}$\\
& & &  $\cup \set{\mca{O}_{((j+1)^2, 1^{2r-2j-2})}}$  \\
\hline
$\frac{r-1}{2}$, & $k$, & $\prod_{a=-k}^{-1} \tau^{a+k}(\mbm{d}_{r, r+2a+1})$, & $\mca{O}_{(2r-2k-1, 2k+1)}$, \\
if $j=r-1$ & $1\lest k \lest \frac{r-1}{2}$ & $1\lest k \lest \frac{r-1}{2}$ & $1\lest k \lest \frac{r-1}{2}$ \\
\hline
\end{tabular}
\end{table}
\vskip 10pt

\begin{table}[H]  
\caption{Type $D_r, r$ even and $S_j, 1\lest j \lest r-1$}
\label{T:D-even}
\vskip 10pt
\renewcommand{\arraystretch}{1.4}
\begin{tabular}{|c|c|c|c|c|c|c|c|c|c|c|}
\hline
$\varphi_{D_r}(j)$ & $\bfa(\mca{C}_{S_j^*})$  &  $X_j$ & $\msc{N}^{\rm spe}(S_j^*)$    \\
\hline
\hline
$\floor{\frac{j+1}{2}}+1$, & $(r-j)(r-j-1)+k$, & $(\prod_{a=-k}^{-1} \mbm{d}_{j+1, j+2+2a}) \cdot w_{S_j^*}$,  & $\big\{\mca{O}_{(2j+1-2k, 2k+1, 1^{2r-2j-2})}:$   \\ 
if $0\lest j \lest r-2$ & $0\lest k \lest \floor{\frac{j+1}{2}}$ & $0\lest k \lest \floor{\frac{j+1}{2}}$ & $0\lest k \lest \floor{\frac{j}{2}} \big\}$\\
& &  & $\cup \set{\mca{O}_{((j+1)^2, 1^{2r-2j-2})}}$  \\
\hline
$\frac{r}{2}$, & $k$, & $\prod_{a=-k}^{-1} \tau^{a+k} (\mbm{d}_{r, r+2a+1})$, & $\mca{O}_{(2r-2k-1, 2k+1)}$, \\
if $j=r-1$ & $1\lest k \lest \frac{r}{2}$ & $1\lest k \lest \frac{r}{2}$ & $1\lest k \lest \frac{r}{2}-1$ \\
& & & and $\begin{cases}{\mca{O}_{(r,r)}^I},&{r	\equiv 0\ mod\ 4}\\{\mca{O}_{(r,r)}^{II}},&{r\equiv 2\ mod\ 4}\end{cases}$ \\ 
\hline
\end{tabular}
\end{table}
\vskip 10pt

\begin{eg}
The graph of $\varphi_{B_6}$ is given as follows:
$$
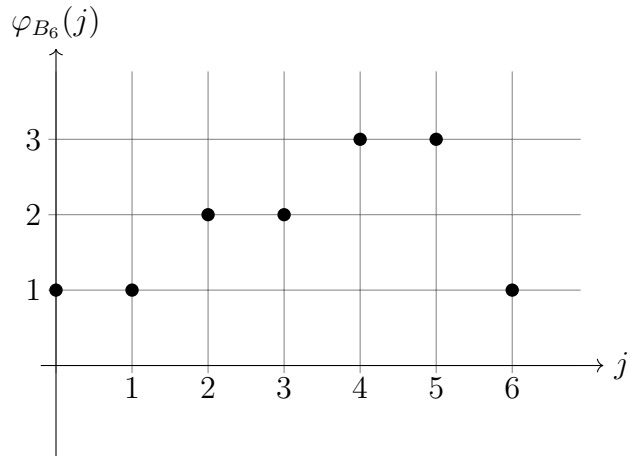
\begin{tikzpicture}[domain=-0.2:4]
 \draw[very thin,color=gray] (-0.1,-0.1) grid (6.9,3.9);
 \draw[->] (-0.2,0) -- (7.2,0) node[right] {$j$};
  \put(0.4,-0.15){$1$} 
    \put(0.8,-0.15){$2$} 
   \put(1.2,-0.15){$3$} 
      \put(1.6,-0.15){$4$}
         \put(2,-0.15){$5$}  
              \put(2.4,-0.15){$6$}
 \draw[->] (0,-1.2) -- (0,4.2) node[above] {$\varphi_{B_6}(j)$};
  \put(-0.15,0.4){$1$} 
  \put(-0.15,0.8){$2$}   
   \put(-0.15,1.2){$3$}  
\put(0,0.4){\circle*{0.07}}  
 \put(0.4,0.4){\circle*{0.07}} 
\put(0.8,0.8){\circle*{0.07}} 
\put(1.2,0.8){\circle*{0.07}} 
\put(1.6,1.2){\circle*{0.07}} 
\put(2,1.2){\circle*{0.07}} 
\put(2.4,0.4){\circle*{0.07}} 
\end{tikzpicture}
$$
On the other hand, we have the graph for $\varphi_{D_6}$ below:
$$
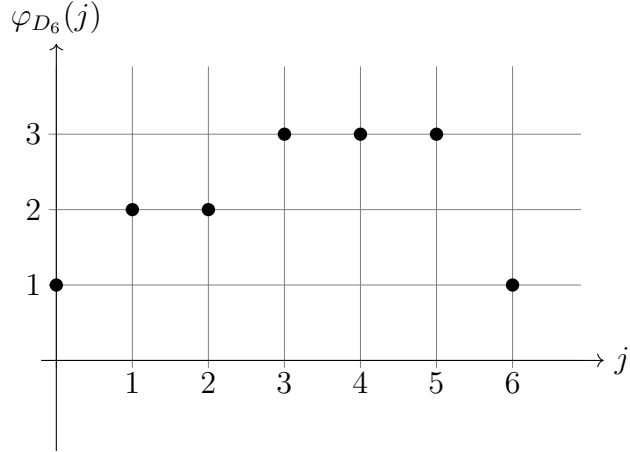
\begin{tikzpicture}[domain=-0.2:4]
 \draw[very thin,color=gray] (-0.1,-0.1) grid (6.9,4.2);
 \draw[->] (-0.2,0) -- (7.2,0) node[right] {$j$};
  \put(0.4,-0.15){$1$} 
    \put(0.8,-0.15){$2$} 
   \put(1.2,-0.15){$3$} 
      \put(1.6,-0.15){$4$}
         \put(2,-0.15){$5$}  
              \put(2.4,-0.15){$6$}
 \draw[->] (0,-1.2) -- (0,4.5) node[above] {$\varphi_{D_6}(j)$};
  \put(-0.15,0.4){$1$} 
  \put(-0.15,0.8){$2$}   
   \put(-0.15,1.2){$3$} 
      \put(-0.15,1.6){$4$}   
\put(0,0.4){\circle*{0.07}}  
 \put(0.4,0.8){\circle*{0.07}} 
\put(0.8,0.8){\circle*{0.07}} 
\put(1.2,1.2){\circle*{0.07}} 
\put(1.6,1.2){\circle*{0.07}} 
\put(2,1.2){\circle*{0.07}} 
\put(2.4,0.4){\circle*{0.07}} 
\end{tikzpicture}
$$
\end{eg}

\subsection{Proof of Theorem \ref{T:ABD}}

\def\ll{\underset {L}{\leq}}
\def\rl{\underset {R}{\leq}}
\def\lr{\rl}
\def\lrl{\underset {LR}{\leq}}
\def\llr{\lrl}
\def\el{\underset {L}{\sim}}
\def\er{\underset {R}{\sim}}
\def\elr{\underset {LR}{\sim}}


The proof follows from a case by case discussion. Throughout this subsection, for every $1\lest i \lest r$, we retain the notation
$$s_i:= s_{\alpha_i} \in W$$
for every simple reflection, where $\alpha_i$ for each type is labelled as in Table \ref{T:ABDG}.

\begin{prop} \label{P:A}
Let $\Delta$ be of type $A_{r}$. Then the $\bfa$-function takes the constant value $j(j+1)/2$ on $\mca{C}_{S_j}$ for every $0\lest j \lest r$. Also, $\mfr{C}^{\rm LR}(S_j) = \set{\mca{C}}$ and $\msc{N}^{\rm spe}(S_j)= \set{\mca{O}_{(r-j+1, 1^j)}}$.
\end{prop}
\begin{proof}
By Lemma \ref{L:YN12}, we know that 
$$w_{S_j^*}w_\Delta \lest_{\rm R} x \lest_{\rm R} w_{S_j} \text{ and } \bfa(w_{S_j^*}w_\Delta) \gest \bfa(x)\gest \bfa(w_{S_j})$$
for every $x\in \mca{C}_{S_j}$.
Thus it suffices to prove $\bfa(w_{S_j^*}w_\Delta)=\bfa(w_{S_j})$.
Let $w_{r-j+1, r-j+2, \cdots, r}$ be the longest element of the parabolic Weyl subgroup generated by the simple reflections 
$$\set{s_{r-j+1}, s_{r-j+2}, \cdots, s_r }.$$ Then by \cite[\S 10.5 and Theorem 5.4]{Lus-CAW1}, we have 
$w_{S_j^*}w_\Delta \sim_{\rm L} w_{r-j+1, r-j+2, \cdots, r}$ and 
$$\bfa(w_{S_j^*}w_\Delta)=\bfa(w_{r-j+1, r-j+2, \cdots, r})=\bfa(w_{S_j})=\frac{j(j+1)}{2}.$$

This shows $\mfr{C}^{\rm LR}(S_j) = \set{\mca{C}}$ is a singleton. The last assertion follows from applying the Robinson--Schensted insertion algorithm to $w_{S_j}$ (see \cite[\S 3.1]{Sag01-B} or \cite[\S 4.1]{BX19}), from which we obtain the nilpotent orbit  $\mca{O}_{(r-j+1, 1^j)}$ associated with $\sigma_\mca{C}$.
\end{proof}

The above Proposition gives Table \ref{T:A}.

\begin{prop} \label{P:B}
Let $\Delta$ be of type $B_{r}$ and $1\lest j \lest r-1$. Then the $\bfa$-function values on $\mca{C}_{S_j^*}$ are 
$$\set{(r-j)^2+k: \ 0\lest k \lest \floor{\frac{j}{2}}}.$$
Also, $\mfr{C}^{\rm LR}(S_j^*) = \set{\mca{C}_k: \ 0\lest k \lest \floor{\frac{j}{2}}}$ and the corresponding set of special nilpotent orbits is 
$$\msc{N}^{\rm spe}(S_j^*)=\set{\mca{O}_{(2j+1-2k, 2k+1,1^{2r-2j-1})}: \ 0\lest k \lest \floor{\frac{j}{2}}}.$$
\end{prop}
\begin{proof}
First, we have 
$$\bfa(w_{S_j^*})=(r-j)^2$$
by \S 1.3 and Corollary 1.9(d) in \cite{Lus-CAW2} with $w_{S_j^*}$ corresponding to the special partition $(2j+1, 1^{2r-2j})$ by Robinson--Schensted insertion algorithm. Meanwhile, since 
$$w_{S_j}w_\Delta \sim_{\rm L} 
\begin{cases}
s_1s_3\cdots s_{j-3}s_{j-1}w_{S_j^*} & \text{ if $j$ is even},\\
s_2s_4\cdots s_{j-3}s_{j-1}w_{S_j^*} & \text{ if $j$ is odd},
\end{cases}
$$
we have 
$\bfa(w_{S_j}w_\Delta)=(r-j)^2+\floor{\frac{j}{2}}$ with $w_{S_j}w_\Delta$ corresponding to the special nilpotent orbit 
$$\mca{O}_{ (2j+1-2\floor{\frac{j}{2}}, 2\floor{\frac{j}{2}}+1, 1^{2r-2j-1}) } \in \msc{N}^{\rm spe}(S_j^*).$$ 

By Lemma \ref{L:YN12}, we know that $\bfa(w_{S_j^*})\lest \bfa(x)\lest \bfa(w_{S_j}w_\Delta)$ for every $x\in \mca{C}_{S_j^*}$. It is easy to check that there are exactly $\floor{\frac{j}{2}}-1$ special partitions properly lying between $(2j+1, 1^{2r-2j})$ and $(2j+1-2\floor{\frac{j}{2}}, 2\floor{\frac{j}{2}}+1, 1^{2r-2j-1})$. More precisely, the special partition associated to any two-sided cell $\mca{C} \in \mfr{C}^{\rm LR}(S_j^*)$ is of the form $(2j+1-2k, 2k+1, 1^{2r-2j-1})$ with the value $\bfa(\mca{C})$ equal to $(r-j)^2+k$, where $0\lest k\lest \floor{\frac{j}{2}}$.

Now it suffices to show there are $\floor{\frac{j}{2}}+1$ elements in $\mca{C}_{S_j^*}$ with distinct $\bfa$-function values. For every $1\lest i \lest r$ and $1\lest q \lest r-1$ we write
$$\mbm{b}_{i,q}:=s_i s_{i+1} ... s_{r-1} s_r s_{r-1} ... s_{q+1} s_q \in W.$$
and $\mbm{b}_{r,r}:=s_r$ as in \eqref{D:bb}. Set 
$$B_k=\left( \prod_{a=-k}^{-1} \mbm{b}_{j+1, j+1+2a} \right) \cdot w_{S_j^*}$$ 
with $1\lest k\lest \floor{\frac{i}{2}}$  and $B_0:=w_{S_j^*}$.
Then for every $0\lest k\lest \floor{\frac{j-2}{2}}$, we have 
$$B_{k+1}= \mbm{b}_{i+1, i-2k-1} \cdot B_k$$
and 
$$\bfa(B_{k+1})\gest \bfa(s_{j-2k-1} \cdot B_k)=\bfa(B_k)+1$$ 
by \cite[Theorem 5.4 and \S 10.5]{Lus-CAW1} and \cite[Corollary 1.9(d)]{Lus-CAW2}. In fact, it follows from \cite[Definition 8.9]{BMW} and \cite[Theorem 1.3]{BXX23} that $\bfa(B_{k})=(r-j)^2+k$ and the corresponding special nilpotent orbit is $\mca{O}_{(2j+1-2k, 2k+1, 1^{2r-2j-1})}$. This concludes the proof.
\end{proof}

\begin{prop} \label{P:D}
Let $\Delta$ be of type $D_{r}$ and consider $S_j$ with $1\lest j \lest r-1$.
\begin{enumerate}
\item[(i)] Assume $1\lest j\lest r-2$. Then the $\bfa$-function values on $\mca{C}_{S_j^*}$ are 
$$\set{(r-j)(r-j-1)+k:\ 0\lest k\lest \floor{\frac{j+1}{2}}}.$$
Moreover, one has $\mfr{C}^{\rm LR}(S_j^*)=\set{\mca{C}_k: \ 0\lest k\lest \floor{\frac{j+1}{2}}}$ with the associated special nilpotent orbits
$$\msc{N}^{\rm spe}(S_j^*)= \set{\mca{O}_{(j+1,j+1, 1^{2r-2j-2})}} \cup {\set{ \mca{O}_{(2j+1-2k, 2k+1,1^{2r-2j-2})}: \ 0\lest k\lest \floor{\frac{j}{2}} }}.$$
(Note that if $j$ is even, then the orbit $\mca{O}_{(j+1,j+1, 1^{2r-2j-2})}$ already appears in the right set of the above union.)
\item[(ii)] Assume $j=r-1$. Then the $\bfa$-function values on the $\mca{C}_{S_{r-1}^*}$ are
$$\set{k: 1\lest k \lest \floor{r/2}},$$
and $\mfr{C}^{\rm LR}(S_{r-1}^*) = \set{\mca{C}_k: \ 1\lest k \lest \floor{r/2}}$. Moreover, one has
$$
\msc{N}^{\rm spe}(S_{r-1}^*) =
\begin{cases}
\set{ \mca{O}_{(2r-2k-1, 2k+1)}: \ 1\lest k \lest \floor{\frac{r}{2}} } & \text{ if $r$ is odd},\\
\set{ \mca{O}_{(2r-2k-1, 2k+1)}: \ 1\lest k \lest r/2-1} \cup \set{\mca{O}_{(r,r)}^I} & \text{ if $r\equiv 0\ mod\ 4$}\\
\set{ \mca{O}_{(2r-2k-1, 2k+1)}: \ 1\lest k \lest r/2-1} \cup \set{\mca{O}_{(r,r)}^{II}} & \text{ if $r\equiv 2\ mod\ 4$}.
\end{cases}
$$
\end{enumerate}
\end{prop}
\begin{proof}
For (i), similar to the proof of Proposition \ref{P:B}, it is easy to get 
$$\bfa(w_{S_j^*})=(r-j)(r-j-1)$$ with $w_{S_j^*}$ corresponding to the special nilpotent orbit $\mca{O}_{(2j+1,1^{2r-2j-1})}$. Also, $\bfa(w_{S_j}w_\Delta)=(r-j)(r-j-1)+\floor{\frac{j+1}{2}}$, while $w_{S_j}w_\Delta$ corresponds to $\mca{O}_{(j+1, j+1,1^{2r-2j-2})}$. These follow from \cite[8.11]{BMW} and \cite[Theorem 1.3]{BXX23}, for example.

There are exactly $\floor{\frac {j-1}{2}}$ special partitions between $(2j+1, 1^{2r-2j-1})$ and $(j+1, j+1,1^{2r-2j-2})$. Thus, it suffices to show there are $\floor{\frac {j+1}{2}}+1$ elements in $\mca{C}_{S_j^*}$ with distinct $\bfa$-function values. 
For every $1\lest i \lest r$ and $1\lest q \lest r-2$, define
$$\mbm{d}_{i,q}:=s_i s_{i+1} ... s_{r-1} s_r s_{r-2} s_{r-3} ... s_{q+1} s_q \in W$$
as in \eqref{D:dd} and set $\mbm{d}_{r,r-1}:=\mbm{d}_{r,r}:=s_r \in W$.
Let $\tau\in {\rm Aut}(W)$ be the unique element such that 
$$\tau(s_r) = s_{r-1}, \quad \tau(s_l)=s_l$$
for every $1\lest l\lest r-2$. 

When $1\lest j\lest r-2$ and $S_j^*=\set{s_{j+1}, s_{j+2}, \cdots, s_{r}}$, for every $1\lest k\lest \floor{\frac{j}{2}}$ we set 
$$D_k:=\left( \prod_{a=-k}^{-1} \mbm{d}_{j+1, j+2+2a} \right) \cdot w_{S_j^*}$$ 
and $D_0:=w_{S_j^*}$.
Then for every $1\lest k\lest \floor{\frac{i-1}{2}}$, we have 
$$D_{k+1}= \mbm{d}_{j+1, j-2k} \cdot D_k.$$
Also, $\bfa(D_{k+1})\gest \bfa(s_{j-2k}D_k)\gest \bfa(D_k)+1$ by Theorem 5.4 and Corollary 5.5 of \cite{Lus-CAW1} and Corollary 1.9(d) of \cite{Lus-CAW2}.

For $D_1=\mbm{d}_{j+1, j}D_0=\mbm{d}_{j+1, j}w_{S_j^*}$, we get from \cite[\S 1.3]{Lus-CAW2} that 
$$\bfa(\mbm{d}_{j+1, j}D_0)\lest \bfa(w_{S_j^*})+1.$$ 
On the other hand,  we have
$$\mbm{d}_{j+1, j}D_0 \sim_{\rm L} \mbm{d}_{j+1, j-1} \mbm{d}_{j+1, j} D_0 =\mbm{d}_{j+1, j-1} D_0 \mbm{d}_{j+1, j}$$ 
by \cite[\S 10.5]{Lus-CAW1} and thus
$$\bfa(\mbm{d}_{j+1, j-1}d_{j+1, j}D_0)\gest \bfa(\mbm{d}_{j+1, j-1}D_0)\gest \bfa(s_{j-1} w_{S_j^*})=\bfa(w_{S_j^*})+1$$ 
by \cite[Theorem 5.4]{Lus-CAW1} and \cite[\S 1.3 and Corollary 1.9(d)]{Lus-CAW2}. This gives $\bfa(\mbm{d}_{j+1, j}D_0)=\bfa(w_{S_j^*})+1$.
In fact, using \cite[Definition 8.11]{BMW} and \cite[Theorem 1.3]{BXX23}, we have $\bfa(D_{k})=(r-j)(r-j-1)+k$ and the corresponding special nilpotent orbit is $\mca{O}_{(2j+1-2k, 2k+1, 1^{2r-2j-2})}$.

For (ii), when $j=r-1$ and $S_j^*=\set{s_r}$, we set 
$$D'_k=\prod_{a=-k}^{-1}\tau^{a+k}(\mbm{d}_{r, r+1+2a})$$
for every $1\lest k\lest \floor{\frac{r}{2}}$.
Then for any $1\lest k\lest \floor{\frac{r-2}{2}}$, we have $D'_{k+1}=
\mbm{d}_{r, r-2k-1}\cdot \tau(D'_k)$ and 
$$\bfa(D'_{k+1})\gest \bfa(s_{r-2k-1}\cdot \tau(D'_k) )\gest \bfa(D'_k)+1$$
by Theorem 5.4 and Corollary 5.5 of \cite{Lus-CAW1} and Corollary 1.9(d) of \cite{Lus-CAW2} and the fact $\bfa(\tau(x))=\bfa(x)$ for every $x\in W$.
Again, using \cite[Definition 8.11]{BMW} and \cite[Theorem 1.3]{BXX23}, we have $\bfa(D'_k)=k$ and the corresponding special nilpotent orbit is $\mca{O}_{(2r-2k-1, 2k+1)}$ for $1\lest k\lest \floor{\frac{r}{2}}-1$. Also, $\bfa(D'_{\floor{\frac{r}{2}}})=\floor{\frac{r}{2}}$ with the corresponding special nilpotent being $\mca{O}_{(r, r)}$ when $r$ is odd, $\mca{O}_{(r, r)}^I$ when $r\equiv 0\ mod\ 4$ and $\mca{O}_{(r, r)}^{II}$ when $r\equiv 2\ mod\ 4$ (see \cite[Lemma 7.10 and Remark 7.11]{BMW}). This completes the proof.
\end{proof}

\subsection{Further result for type $A_r$} \label{SS:Amore}
In the case of type $A_r$, we also consider 
\begin{equation} \label{F:Tj}
T_j:=S_j - \set{\alpha_1}= \set{\alpha_2, \alpha_3, ..., \alpha_j}
\end{equation} 
for $0\lest j \lest r$ and have the following result which complements Theorem \ref{T:ABD}.

\begin{prop} \label{P:Amore}
Let $\Delta$ be of type $A_r$,
and consider $T_j$ with $2\lest j\lest r-1$. 
We have  $\val{\mfr{C}^{\rm LR}(T_j)}=2$, and the two nilpotent orbits associated with $\mfr{C}^{\rm LR}(T_j)$ are $\mca{O}_{(r-j+2,1^{j-1})}$ and $\mca{O}_{(r-j+1,2,1^{j-2})}$ in this case.
\end{prop}
\begin{proof}
For the first assertion, in view of Lemma \ref{L:YN12}, we only need to show that  $\bfa(w_{T_j^*} \cdot w_\Delta) - \bfa(w_{T_j}) =1$.  Since $w_{T_j}$ is a distinguished involution, we get $\bfa(w_{T_j})=l(w_{T_j})=\frac{j(j-1)}{2}$ immediately by \cite[Proposition 2.4]{Lus-CAW1} and \cite[1.3]{Lus-CAW2} . Moreover, since
$$w_{T_j^*} \cdot w_\Delta \sim_{\rm LR} w_{1,(r+2-j), (r+3-j),\cdots, r},$$
which is also a distinguished involution, we get $\bfa(w_{T_j^*} \cdot w_\Delta)=l(w_{1,(r+2-j), (r+3-j),\cdots, r})=\frac{j(j-1)}{2}+1$. 

The second assertion follows from applying the Robinson--Schensted insertion algorithm to $w_{T_j}$ and $w_{T_j^*} \cdot w_\Delta$ (see \cite[\S 3.1]{Sag01-B} or \cite[\S 4.1]{BX19}), from which we obtain the corresponding nilpotent orbits  $\mca{O}_{(r-j+2,1^{j-1})}$ and $\mca{O}_{(r-j+1,2,1^{j-2})}$ respectively.
\end{proof}

%
%


\section{The Whittaker polynomial $\mca{P}_{G,S}(X)$} \label{S:WhP}

Henceforth, let $\mbf{G}$ be a split simply-connected almost simple group over a $p$-adic field $F$.  Let $Y$ be the cocharacter lattice of $\mbf{G}$, which is also equal to the coroot lattice in this case.  We use $W$ to denote the Weyl group of the coroot system of $\mbf{G}$, which acts naturally on $Y$. Thus, $W$ is generated by the simple reflections $\set{s_{\alpha^\vee}}$ for all simple coroot $\alpha^\vee \in \Delta^\vee$. Since we can naturally identify $W$ with the Weyl group of the root system by $s_\alpha \leftrightarrow s_{\alpha^\vee}$, we hope the usage of the notation $W$ does not cause confusion.

Consider the unique $W$-invariant quadratic form
$$Q: Y \longrightarrow \Z$$
such that $Q(\alpha^\vee) = 1$ for every short coroot $\alpha^\vee$ of $G$. Assume $F^\times$ contains the full  group $\mu_n$ of $n$-th roots of unity. We have an $n$-fold central cover of $G=\mbf{G}(F)$ as in 
$$\begin{tikzcd}
\mu_n \ar[r, hook] & \wt{G}^{(n)} \ar[r, two heads] & G.
\end{tikzcd}$$
For every subgroup $H\subseteq G$, we write $\wt{H}^{(n)}$ or simply $\wt{H}$ for the $n$-fold cover of $H$ obtained from restricting $\wt{G}^{(n)}$ to $H$. Fix an embedding 
$$\epsilon: \mu_n \into \C^\times,$$
we write $\Irr_\epsilon(\wt{G}^{(n)})$ for the set of isomorphism classes of irreducible $\epsilon$-genuine representations of $\wt{G}^{(n)}$, where $\mu_n$ acts via $\epsilon$.  For more detailed discussion on covering groups, see \cite{BD, GG, We6}. 
\subsection{Whittaker models for oasitic covers}
Below we follow closely the notation and exposition in \cite{GGK1} for introducing Whittaker model for representations of $\wt{G}^{(n)}$.

Let $B^- = T U^- \subset G$ be the opposite Borel subgroup associated with $-\Delta$, where $U^-$ is the opposite unipotent subgroup. Since $\wt{G}^{(n)}$ splits uniquely over $U^-$, we view $U^-$ as a subgroup of $\wt{G}^{(n)}$. 
Let $\psi: F \to \C^\times$ be a character of conductor $\mfr{p}_F$. We view it as a character of $U^-$ by requiring that
$$\psi(e_\alpha(x)):=\psi(x)$$
for every $\alpha \in -\Delta$, where $e_\alpha: F \to U_\alpha$ is a fixed pinning for the one-parameter subgroup $U_\alpha \subset G$ associated with $\alpha$.  Write
$$\mca{V}_\epsilon:={\rm ind}_{\mu_n \times U^-}^{\wt{G}^{(n)}} (\epsilon \times \psi)$$ for the $\epsilon$-genuine Gelfand--Graev representation of $\wt{G}^{(n)}$ with left action given by $(g\cdot f)(x):=f(xg), f \in \mca{V}_\epsilon$.  For every $\pi\in \Irr_\epsilon(\wt{G}^{(n)})$, its Whittaker model is
\begin{equation} \label{D:Wh}
\Wh_\psi^{G,n}(\pi):=\Hom_{\wt{G}^{(n)}}(\mca{V}_\epsilon, \pi).
\end{equation}
We call $\pi$ generic if $\Wh_\psi^{G,n}(\pi) \ne 0$.

Let $I$ be the Iwahori subgroup with a fixed splitting in $\wt{G}^{(n)}$. Let $\HH_{\bepsilon}(\wt{G}^{(n)}, I)$ be the $\bepsilon$-genuine Iwahori--Hecke algebra. Suppose $\pi \in \Irr_\epsilon(\wt{G}^{(n)})$ is Iwahori-spherical and corresponds to $\tau \in \Irr(\HH_{\bepsilon}(\wt{G}^{(n)}, I))$; we may write $\pi_\tau$ for $\pi$ to highlight this correspondence.  Then 
$$\Wh_\psi^{G,n}(\pi_\tau) \simeq \Hom_{\HH_{\bepsilon}(\wt{G}^{(n)}, I)}( (\mca{V}_\epsilon)^I, \tau).$$

In this paper, we concentrate exclusively on the oasitic covers $\wt{G}^{(n)}$ of $G$ as defined in \cite[\S 6.1]{GGK1}). For fixed $G$, we also call $n$ oasitic if $\wt{G}^{(n)}$ is an oasitic cover. The oasitic numbers are ``very good" in the sense of \cite{Som97}, as we mentioned in \S \ref{S:Intro}. They are tabulated in Table \ref{T:oas}.

\begin{table}[H]  
\caption{Oasitic covers}
\label{T:oas}
\vskip 10pt
\renewcommand{\arraystretch}{1.4}
\begin{tabular}{|c|c|c|c|c|c|c|c|c|c|c|}
\hline
 & $\SL_{r+1}$  &  $\Spin_{2r+1}$ & $\Sp_{2r}$  & $\Spin_{2r}$    \\
\hline
\hline
oasitic & $\gcd(n, r+1)=1$ & $n$ odd & $n$ odd & $n$ odd   \\ 
\hline
\end{tabular}
\vskip 10pt

\begin{tabular}{|c|c|c|c|c|c|c|c|c|c|c|}
\hline
 & $E_6$ & $E_7$& $E_8$ & $F_4$ & $G_2$   \\ 
\hline
\hline
oasitic & $2, 3\nmid n$ & $2, 3 \nmid n$& $2, 3, 5 \nmid n$ & $2, 3 \nmid n$ & $2, 3\nmid n$   \\  
\hline
\end{tabular}
\end{table}
\vskip 10pt

 Several properties of oasitic covers are as follows. First, we observe that for arbitrary $n$, the set $\msc{X}_n:=Y/nY$ affords a natural permutation representation 
$$\eta_{\msc{X}_n}: W \longrightarrow {\rm Perm}(\msc{X}_n)$$
given by $\eta_{\msc{X}_n}(w)(y):= w(y)$.  For oasitic covers, the module structure of $(\mca{V}_\epsilon)^I$ over $\HH_{\bepsilon}(\wt{G}^{(n)}, I)$, is given by (see \cite[Theorem 1.2]{GGK1})
$$(\mca{V}_\epsilon)^I \simeq V_{\msc{X}_n} \otimes_{\HH_W} \HH_{\bepsilon}(\wt{G}^{(n)}, I),$$
where $V_{\msc{X}_n}$ is a deformation of $\eta_{\msc{X}_n} \otimes \varepsilon_W$ such that
\begin{equation} \label{E:dform}
(V_{\msc{X}_n})_{q\to 1} \simeq \eta_{\msc{X}_n} \otimes \varepsilon_W
\end{equation}
as $W$-representations. Here $\HH_W \subset \HH_{\bepsilon}(\wt{G}^{(n)}, I)$ is the finite Weyl subalgebra, a deformation of $\C[W]$. Moreover, $(-)_{q\to 1}$ is the operation of sending $\HH_W$-modules to $W$-modules, and is an isometry, see \cite[Proposition 10.11.4]{Car}.

Hence, for every Iwahori-spherical $\pi_\tau \in \Irr_{\epsilon}(\wt{G}^{(n)})$, we have
\begin{equation} \label{cal-Wd}
\Wh_{\psi}^{G,n}(\pi_\tau) \simeq \Hom_{\HH_W}(V_{\msc{X}_n}, \sigma|_{\HH_W}) \simeq \Hom_{W}(\eta_{\msc{X}_n} \otimes \varepsilon_W, (\tau|_{\HH_W})_{q\to 1}).
\end{equation}
The dimension of $\Wh_\psi^{G,n}(\pi_\tau)$ is called the Whittaker dimension of $\pi_\tau$.

\subsection{Regular principal series} \label{SS:rps}
Now we specialize $\pi_\tau$ to be an irreducible constitute which lies in the same genuine principal series as the covering analogue of the Steinberg representation, see \cite{Rod4, Ga6}. We study exclusively in the remaining of this paper the Whittaker dimension of such $\pi_\tau$. More precisely, denote by $Z(\wt{T}) \subseteq \wt{T}$ the center of the covering torus $\wt{T}$. Consider a so-called exceptional unramified genuine character 
$$\chi_{\rm ex}: Z(\wt{T}) \to \C^\times$$
satisfying
$$\chi_{\rm ex}(\wt{h}_\alpha(\varpi^{n_\alpha})) = q^{-1}$$
for every $\alpha \in \Delta$, where $n_\alpha:=n/\gcd(n, Q(\alpha^\vee))$. The group $W$ acts naturally on $\chi_{\rm ex}$ with trivial stabilizer. 
Thus, $\chi$ is a regular genuine central character and the irreducible constituents of $I(\chi_{\rm ex})$ are all multiplicity-free. 
We denote its Jordan--Holder set by ${\rm JH}(I(\chi_{\rm ex}))$. 

Let $\hat{Y}:=\Hom_\Z(Y,\Z)$ be the character lattice of $G$.
Denote by
$$\msc{C}(\hat{Y} \otimes \R)$$
the set of connected components of 
$$\hat{Y} \otimes \R - \bigcup_{\alpha \in \Delta} {\rm Ker}(\alpha^\vee).$$
For every $S\subseteq \Delta$, we write $S^\vee:=\set{\alpha^\vee}_{\alpha \in S} \subseteq \Delta^\vee$. Clearly, $S \mapsto S^\vee$ gives a canonical bijection between $\msc{P}(\Delta)$ and $\msc{P}(\Delta^\vee)$. 
For every $E \subset \Delta^\vee$, we also write $E^*:=\Delta^\vee - E$. It is clear that
$$(S^*)^\vee = (S^\vee)^*$$
for every $S \subseteq \Delta$.
We denote by $\Phi_-^\vee$ the set of negative coroots.

\begin{prop}[{ \cite{Rod4, Ga6} }] \label{P:Rod}
There are natural bijections between the three sets
$$\msc{P}(\Delta) \longleftrightarrow {\rm JH}(I(\chi_{\rm ex})) \longleftrightarrow \msc{C}(\hat{Y}\otimes \R)$$
denoted by
$$S \leftrightarrow \pi_S \leftrightarrow \Gamma_S,$$
which is given as follows. First, we have
$$\Gamma_S = \set{x\in \hat{Y} \otimes \R: \angb{\alpha^\vee}{x} <0 \text{ if and only if } \alpha \in S}.$$
Second, the representation $\pi_S$ is characterized by its Jacquet module
$$(\pi_S)_U = \bigoplus_{w \in W_S} \delta_B^{1/2} \cdot i({}^{w^{-1}} \chi_{\rm ex}),$$
where  $W_S:= \set{w\in W: \Delta^\vee \cap w(\Phi_-^\vee) = S^\vee} \subset W$ is exactly the set $\mca{C}_{S^\vee} \subset W$ in \eqref{D:C_S} associated with the coroot system $\Delta^\vee$ and $S^\vee \subset \Delta^\vee$.
\end{prop}

There are two special elements $\pi_\emptyset, \pi_\Delta \in {\rm JH}(I(\chi_{\rm ex}))$. First, $\pi_\emptyset$ is the covering analogue of the Steinberg representation and is the unique subrepresentation of $I(\chi_{\rm ex})$. On the other hand, $\pi_\Delta$ is the unique Langlands quotient of $I(\chi_{\rm ex})$, and is often called a theta representation of $\wt{G}^{(n)}$.

\subsection{The polynomial $\mca{P}_{G,S}(X)$}
Let $\tau_S:=(\pi_S)^I \in \Irr(\HH_{\bepsilon}(\wt{G}, I))$ be the irreducible representation associated with $\pi_S \in {\rm JH}(I(\chi_{\rm ex})), S \in \msc{P}(\Delta)$. Then one has
$$(\tau_S|_{\HH_W})_{q\to 1} =\varepsilon_W \otimes \sigma_{S^\vee},$$
which follows from \cite[Lemma 4.8]{Ga6}. Here $\sigma_{S^\vee}$ is the representation of $W$ associated with $\mca{C}_{S^\vee}$ as introduced in \S \ref{SS:sigS}.
 Thus, we have from \eqref{cal-Wd} that
\begin{equation} \label{E:key}
\Wh_\psi^{G,n}(\pi_S) \simeq \Hom_W(\eta_{\msc{X}_n}, \sigma_{S^\vee}).
\end{equation}
This isomorphism is pivotal since it connects the current and subsequent discussion to the content of \S \ref{S:cells}.

Denote by  $\chi_{\msc{X}_n}$ the character of $\eta_{\msc{X}_n}$. For oasitic $\wt{G}^{(n)}$, the value of $\chi_{\msc{X}_n}(w), w\in W$ is computed by Sommers in \cite{Som97}. More precisely, for every $w\in W$, as in \S \ref{S:Intro} we let $l^\sharp(w)$ denote the least number of reflections whose product is $w$ and consider
$$d(w):=\dim (Y\otimes \R)^w,$$
the dimension of the set of fixed points of $w$ in $Y\otimes \R$. One has $d(w)= r -l^\sharp(w)$. It was shown in \cite[Proposition 3.9]{Som97} that
$$\chi_{\msc{X}_n}(w)=n^{d(w)}$$
for every $w\in W$. In view of \eqref{E:exf1}, the character of $\sigma_{S^\vee}$ is always $\Z$-valued. Thus, for fixed $G$ and $S \subseteq \Delta$, the function 
$$\dim \Wh_\psi^{G,n}(\pi_S)=\dim \Hom_W(\eta_{\msc{X}_n}, \sigma_{S^\vee}) \in \Q[n]$$
is a rational polynomial in oasitic $n$ with coefficients depending only on $G$ and $S$.

\begin{dfn} \label{D:wdp}
For fixed simply-connected almost simple $G$ and fixed $S \subseteq \Delta$, we call the degree $r$ polynomial  $\mca{P}_{G,S}(X) \in \Q[X]$ such that $$\mca{P}_{G,S}(n)=\dim \Wh_\psi^{G,n}(\pi_S)$$
for all oasitic $n\in \N$ the Whittaker polynomial associated with $G$ and $S$.
\end{dfn}

As two examples, we have  that (see \cite[Theorem 3.1]{GGK2})
\begin{equation} \label{E:2extr}
\mca{P}_{G,\emptyset}(X)=\val{W}^{-1}\cdot \prod_{j=1}^r (X+m_j), \quad \mca{P}_{G,\Delta}(X)=\val{W}^{-1}\cdot \prod_{j=1}^r (X-m_j),
\end{equation}
where $m_j, 1\lest j \lest r$ are the exponents of the Weyl group. In particular, $\mca{P}_{G,\emptyset}(X)$ and $\mca{P}_{G,\Delta}(X)$ both split over $\Q$ and the roots are invariants determined by $W$. This motivates us to ask the following question as mentioned in \S \ref{S:Intro}:
\begin{enumerate}
\item[$\bullet$] (Q0) For fixed $G$, determine $S \in \msc{P}(\Delta)$ such that $\mca{P}_{G,S}(X) \in \Q[X]$ is a split polynomial over $\Q$.
\end{enumerate}

First, we give some observations:
\begin{lm} \label{L:obs}
\begin{enumerate}
\item[(i)] For every $S \in \msc{P}(\Delta)$, the equality 
$$\mca{P}_{G,S^*}(X) = (-1)^r \cdot \mca{P}_{G,S}(-X) \in \Q[X]$$ holds.
\item[(ii)] For every $S \in \msc{P}(\Delta^\vee) - \set{\emptyset}$, one has $(X-1) | \mca{P}_{G, S}(X)$.
\end{enumerate}
\end{lm}
\begin{proof}
The proof of (i) is essentially the same as given in the proof of \cite[Theorem 3.1]{GGK2}. We have $\varepsilon_W(w) = (-1)^r \cdot (-1)^{d(w)}$ and thus for all oasitic $n$, 
$$\begin{aligned}
\mca{P}_{G,S^*}(n) & = \angb{\eta_{\msc{X}_n}}{\sigma_{(S^*)^\vee}}_W = \angb{\eta_{\msc{X}_n}}{\sigma_{(S^\vee)^*}}_W  \\
& = \frac{1}{\val{W}} \sum_{w\in W} \varepsilon_W(w) \chi_{\sigma_{S^\vee}}(w) \chi_{\msc{X}_n}(w) \\
& = \frac{(-1)^r}{\val{W}} \sum_{w\in W} \chi_{\sigma_{S^\vee}}(w) (-n)^{d(w)} \\
& = (-1)^r \cdot \mca{P}_{G,S}(-n).
\end{aligned} $$
This gives (i).

For (ii), it suffices to show that $\mca{P}_{G, S}(1)=0$ for every $S \ne \emptyset$. But for $n=1$ and thus $\wt{G}^{(n)}=G$ is a linear algebraic group, the Steinberg representation $\pi_\emptyset$ is the only generic constituent in $I(\chi_{\rm ex})$ for $G$. This gives $\mca{P}_{G, S}(1)=0$ for all non-empty $S \subseteq \Delta$.
\end{proof}

\section{Splitting of $\mca{P}_{G,S}(X)$ over $\Q$} \label{S:splWhP} 

We set $\msc{P}(\Delta)_\flat \subseteq \msc{P}(\Delta)$ to be
\begin{equation} \label{D:Pflat}
\msc{P}(\Delta)_\flat:=
\begin{cases}
\set{S_j: 0 \lest j \lest r} \bigcup \set{T_j: 2\lest j \lest r-1} & \text{ for type $A_r$}, \\
\set{S_0, S_1, S_2, S_3} & \text{ for type $B_r$ and $C_r$}, \\
\set{S_0, S_1} & \text{ for type $D_r$}, \\
\set{S_0, S_1} & \text{ for type $G_2$},\\
\set{S_0} & \text{ for type $F_4, E_6, E_7, E_8$}.
\end{cases}
\end{equation}
Also denote 
\begin{equation} \label{D:Pflat*}
\msc{P}(\Delta)_\flat^*:=\set{S^*: \ S \in \msc{P}(\Delta)_\flat}.
\end{equation} 
The goal of this section is to prove the following:

\begin{thm} \label{T:poly}
Let $G$ be simply-connected and almost simple. For every $S \in \msc{P}(\Delta)_\flat \cup \msc{P}(\Delta)_\flat^*$, the Whittaker polynomial $\mca{P}_{G,S}(X) \in \Q[X]$ splits over $\Q$. 
\end{thm}

In view of the relation between $\mca{P}_{G,S}(X)$ and $\mca{P}_{G,S^*}(X)$ in Lemma \ref{L:obs}, it suffices to prove Theorem \ref{T:poly} for $S \in \msc{P}(\Delta)_\flat$. If $G$ is of type $F_4, E_i, 6 \lest i \lest 8$, then the result trivially follows from \eqref{E:2extr}. For other types, we will give a case by case discussion. Thus, Theorem \ref{T:poly} follows from an amalgamation of Propositions \ref{P:A-poly}, \ref{P:BC-poly}, \ref{P:D-poly} and Table \ref{T:G2} below, which actually give more information on the explicit forms of $\mca{P}_{G,S}(X)$.

Before we give a detailed discussion, we note that our proof for type $A_r, B_r, C_r$ relies on the work of Gyoja--Nishiyama--Shimura \cite{GNS99}, which roughly asserts that if $\sigma, \sigma' \in \Irr(W)$ are in the same family, i.e., $\sigma \sim_{\rm LR} \sigma'$ as in Definition \ref{D:LR},  then 
$$\dim \Hom_W(\eta_{\msc{X}_n}, \sigma) = \dim \Hom_W(\eta_{\msc{X}_n}, \sigma').$$
Following the notation in \eqref{E:IrrS},
$$\sigma_{S^\vee} = \bigoplus_{\theta \in \Irr(W)^{\rm spe}_{S^\vee}} \bigoplus_{\substack{\rho \in \Irr(W) \\ \rho^\sharp = \theta}} \mfr{m}(\rho, \sigma_{S^\vee}) \cdot \rho,$$
thus  for $G$ of type $A_r, B_r$ and $C_r$ one has
\begin{equation} \label{E:red}
\mca{P}_{G,S}(n)=\sum_{\theta \in \Irr(W)^{\rm spe}_{S^\vee}} \Big( \sum_{\substack{\rho \in \Irr(W) \\ \rho^\sharp = \theta}} \mfr{m}(\rho, \sigma_{S^\vee}) \Big) \cdot \dim \Hom_W(\eta_{\msc{X}_n}, \theta).
\end{equation}
The value of $\dim \Hom_W(\eta_{\msc{X}_n}, \theta)$ for every $\theta \in \Irr(W)$ (not necessarily special) is already computed in \cite{GNS99}. We see that in order for $\mca{P}_{G,S}(X)$ to be splitting over $\Q$, it is desirable to have $\Irr(W)^{\rm spe}_{S^\vee}$ to be of small size, or equivalently, for $\mfr{C}^{\rm LR}(S^\vee)$ to be small, in view of \eqref{E:bij-s}. 

\subsection{Type $A_r$}
Recall that $S_j:=\set{\alpha_1, ..., \alpha_j}$ and $T_j:=S_j - \set{\alpha_1} = \set{\alpha_2, \alpha_3, ..., \alpha_j}$. First, we have
$$\sigma_{S_j^\vee} = j_{W(S_j^\vee)}^W \varepsilon_{W(S_j^\vee)},$$
where the right hand denotes the Lusztig--Spaltenstein $j$-induction from the parabolic Weyl subgroup $W(S_j)$. Indeed, this just follows from Proposition \ref{P:A} (or equivalently, Table \ref{T:A}) and the Springer correspondence for type $A_r$.

\begin{prop} \label{P:A-poly}
Let $G=\SL_{r+1}$. For $S_j, 0\lest j \lest r$ as above one has
$$\mca{P}_{G,S_j}(X)=\frac{\dim \sigma_{S_j^\vee}}{\val{W}} \cdot \prod_{a=1}^{r-j}(X+a) \cdot \prod_{b=1}^{j}(X-b) \in \Q[X].$$
Also, for $T_{j}$ with $2\lest j \lest r-1$ one has $\mfr{C}^{\rm LR}(T_j^\vee)=\set{\mca{C}_1, \mca{C}_2}$, whose associated special Weyl group representations are $\sigma_{\mca{C}_1}=(r-j+2, 1^{j-1})$ and $\sigma_{\mca{C}_2}=(r-j+1, 2,1^{j-2})$. In this case, we have
$$\mca{P}_{G, T_j}(X)= c_{G, T_j} \cdot (X+ r-j+1)\cdot \prod_{a=1}^{r-j}(X+a) \cdot \prod_{b=1}^{j-1}(X- b) \in \Q[X],$$
where $c_{G,T_j} \in \Q^\times$ depends only on $G$ and $T_j$.
\end{prop}
\begin{proof}
In the standard parametrization of $\Irr(W(A_r))$ using partitions of $r+1$ (see \cite[\S 5.4]{GePf}), the representation $\sigma_{S_j^\vee}$ corresponds to the partition $(r-j+1, 1^j)$, where the correspondence is normalized such that $\mbm{1}_W$ and $\varepsilon_W$ correspond to the partitions $(r+1), (1^{r+1})$ respectively. The result then follows from \cite[Proposition 3.1, (3.3)]{GNS99}. We note that in the notation of loc. cit., the term $t$ should be removed from the formula for $\tau^*(\chi^\alpha; t)$ in (3.3) there.

We consider $T_j$ now. The first assertion follows from Proposition \ref{P:Amore}. This shows that for every irreducible constituent $\sigma \subset \sigma_{T^{\vee}_j}$, one has $\sigma \sim_{\rm LR} \sigma_{\mca{C}_1}$ or $\sigma \sim_{\rm LR} \sigma_{\mca{C}_2}$. We have from \cite[Proposition 3.1]{GNS99} that 
$$\frac{ \angb{\eta_{\msc{X}_n}}{\sigma} }{ \dim \sigma } = \frac{ \angb{\eta_{\msc{X}_n}}{\sigma_{\mca{C}_i}}  }{ \dim \sigma_{\mca{C}_i} }$$
whenever $\sigma \sim_{\rm LR} \sigma_{\mca{C}_i}, 1\lest i \lest 2$, and moreover
$$\angb{\eta_{\msc{X}_n}}{\sigma_{\mca{C}_1}} =\frac{\dim \sigma_{\mca{C}_1}}{\val{W}} \cdot \prod_{a=1}^{r-j+1}(n+a) \cdot \prod_{b=1}^{j-1}(n- b)$$
and 
$$\angb{\eta_{\msc{X}_n}}{\sigma_{\mca{C}_2}} =\frac{\dim \sigma_{\mca{C}_2}}{\val{W}} \cdot \prod_{a=0}^{r-j}(n+a) \cdot \prod_{b=1}^{j-1}(n- b).$$
This gives the formula for $\mca{P}_{G,T_j}(X)$ in view of  \eqref{E:red}.
\end{proof}

\subsection{Type $B_r$ and $C_r$}
We consider $G=\Spin_{2r+1}$ and $\Sp_{2r}$ with $\Delta$ of type $B_r$ and $C_r$ respectively. In view of Lemma \ref{L:obs}, it suffices to consider either $S_j$ or $S_j^*$ with $j\in \set{1, 2, 3}$. Note that the Weyl groups of type $B_r$ and $C_r$ are naturally isomorphic and every $\sigma \in \Irr(W)$ is parametrized by bipartition $(\xi; \eta)$  such that $\val{\xi} + \val{\eta} = r$, see \cite[\S 11.4]{Car}.

\begin{prop} \label{P:BC-poly}
Let $G$ be $\Spin_{2r+1}$ or $\Sp_{2r}$ and consider $S_j, j \in \set{1, 2, 3} \subseteq \Delta$.
\begin{enumerate}
\item[(i)] If $j=1$, then $\mfr{C}^{\rm LR}((S_1^\vee)^*)=\set{\mca{C}}$ is a singleton and the special Weyl group representation $\sigma_\mca{C}$ associated with $\mca{C}$ has bipartition $(1; 1^{r-1})$. In this case, we have
$$\mca{P}_{G, S_1^*}(X) = c_{G, S_1}\cdot (X+1) \prod_{a=1}^{r-1}(X-(2a-1)) \in \Q[X],$$
where $c_{G, S_1} \in \Q^\times$ depends only on $G$ and $S_1$.
\item[(ii)] If $j=2$, then $\mfr{C}^{\rm LR}((S_2^\vee)^*)=\set{\mca{C}_1, \mca{C}_2}$, whose associated special Weyl group representations are $\sigma_{\mca{C}_1}=(2; 1^{r-2})$ and $\sigma_{\mca{C}_2}=(1; 2,1^{r-2})$. In this case, we have
$$\mca{P}_{G, S_2^*}(X)= (c_{G,S_2} \cdot X + d_{G,S_2}) \cdot (X+1) \cdot \prod_{a=1}^{r-2}(X-(2a-1)) \in \Q[X],$$
where $c_{G,S_2} \in \Q^\times, d_{G,S_2} \in \Q$ depend only on $G, S_2$.
\item[(iii)] If $j=3$, then $\mfr{C}^{\rm LR}((S_3^\vee)^*)=\set{\mca{C}_1, \mca{C}_2}$, whose associated special Weyl group representations are $\sigma_{\mca{C}_1}=(3; 1^{r-3})$ and $\sigma_{\mca{C}_2}=(2; 1^{r-2})$. In this case, we have
$$\mca{P}_{G, S_3^*}(X)= (c_{G,S_3} \cdot X + d_{G,S_3}) \cdot (X+1) (X+3) \cdot \prod_{a=1}^{r-3}(X-(2a-1)) \in \Q[X],$$
where $c_{G,S_3} \in \Q^\times, d_{G,S_3} \in \Q$ depend only on $G, S_3$.
\end{enumerate}
\end{prop}
\begin{proof}
Note that for both type $B_r$ and $C_r$, the number $n$ is oasitic if and only if $n$ is odd, see Table \ref{T:oas}.
In view of the canonical identification $W(B_r) \simeq W(C_r)$ and \eqref{E:key}, we have
$$\dim \Wh_\psi^{\Spin_{2r+1}, n}(\pi_{S_j^\vee}) = \dim \Wh_\psi^{\Sp_{2r}, n}(\pi_{S_j^\vee})$$
for every $0\lest j \lest r$ and all oasitic $n$. Thus, 
$$\mca{P}_{\Spin_{2r+1}, S_j}(X) = \mca{P}_{\Sp_{2r}, S_j}(X) \in \Q[X]$$
for all $j$. Note also that $\mfr{C}^{\rm LR}(S_j)$ is independent of whether $\Delta$ is of type $B_r$ or $C_r$. Therefore, it suffices to consider the case when $G=\Sp_{2r}$ and thus $\Delta^\vee$ is of type $B_r$. We assume this for the rest of the proof.

We first show (i). The fact that $\mfr{C}^{\rm LR}((S_1^\vee)^*)$ is a singleton set $\set{\mca{C}}$ follows from Table \ref{T:B}. The two-sided cell is then the unique one containing $s_{\alpha_1^\vee} w_{\Delta^\vee}$. It also follows from loc. cit. that
$$\sigma_\mca{C} = (1; 1^{r-1}) \in \Irr(W)^{\rm spe}.$$
Let $\sigma \subset \sigma_{S_1^\vee}$ be any irreducible constituent. Then $\sigma \sim_{\rm LR} \sigma_\mca{C}$, and it follows from \cite[Proposition 3.3]{GNS99} that
$$\frac{ \angb{\eta_{\msc{X}_n}}{\sigma}}{ \dim \sigma } = \frac{ \angb{\eta_{\msc{X}_n}}{\sigma_\mca{C}}}{ \dim \sigma_\mca{C} }  = \frac{1}{\val{W}}  \cdot (n+1)\cdot \prod_{a=1}^{r-1}(n-(2a-1)),$$
the result for (i) follows in view of \eqref{E:red}.

The proof of (ii) and (iii) follows from the same idea for (i). For $S_2$, it follows from Table \ref{T:B} that there are two two-sided cells that intersect nontrivially with $\mfr{C}^{\rm LR}(S_2^\vee)$, say $\mca{C}_1, \mca{C}_2$. Here we may assume that $\mca{C}_1$ is the unique one containing $w_{(S_2^\vee)^*}$ and $\mca{C}_2$ the one containing $w_{S_2^\vee} w_{\Delta^\vee}$. Again, the special representations associated to $\mca{C}_1, \mca{C}_2$ are
$$\sigma_{\mca{C}_1}=(2; 1^{r-2}), \quad \sigma_{\mca{C}_2} = (1; 2,1^{r-2}).$$
For every irreducible constituent $\sigma \subset \sigma_S$, one has $\sigma \sim_{\rm LR} \sigma_{\mca{C}_1}$ or $\sigma \sim_{\rm LR} \sigma_{\mca{C}_2}$. We have from \cite[Proposition 3.3]{GNS99} that 
$$\frac{ \angb{\eta_{\msc{X}_n}}{\sigma} }{ \dim \sigma } = \frac{ \angb{\eta_{\msc{X}_n}}{\sigma_{\mca{C}_i}}  }{ \dim \sigma_{\mca{C}_i} }$$
whenever $\sigma \sim_{\rm LR} \sigma_{\mca{C}_i}, 1\lest i \lest 2$, and moreover
$$\angb{\eta_{\msc{X}_n}}{\sigma_{\mca{C}_1}} =\frac{\dim \sigma_{\mca{C}_1}}{\val{W}} \cdot (n+3) \cdot (n+1) \cdot \prod_{a=1}^{r-2}(n-(2a-1))$$
and 
$$\angb{\eta_{\msc{X}_n}}{\sigma_{\mca{C}_2}} =\frac{\dim \sigma_{\mca{C}_2}}{\val{W}} \cdot (n+1) \cdot (n+1) \cdot \prod_{a=1}^{r-2}(n-(2a-1)).$$
This gives the formula for $\mca{P}_{G,S_2^*}(X)$ in view of  \eqref{E:red}.

For (iii), the two two-sided cells $\mca{C}_1$ and $\mca{C}_2$ are associated with $w_{(S_3^\vee)^*}$ and $w_{(S_3^\vee)} w_{\Delta^\vee}$ respectively. Then the computation of $\sigma_{\mca{C}_i}, i =1, 2$ and the rest of the argument are in complete parallel as case (ii). We omit the details.
\end{proof}

We note that the above formula for $\mca{P}_{G, S_j^*}(X), 1\lest j \lest 3$ immediately gives that for  $\mca{P}_{G, S_j}(X)$ in view of the ``functional equation" in Lemma \ref{L:obs} (i), which in fact also follows from \cite[Proposition 1.4]{GNS99}.

\subsection{Type $D_r$} For type $D_r$, it suffices to consider $S_1=\set{\alpha_1}$. However, we can not argue as in the $B_r$, since it is possible that $\sigma_1 \sim_{\rm LR} \sigma_2$ in $\Irr(W)$, but $\angb{\eta_{\msc{X}_n}}{\sigma_1}/\dim \sigma_1 \ne  \angb{\eta_{\msc{X}_n}}{\sigma_2}/\dim \sigma_2$, see the discussion on \cite[Page 17]{GNS99}. Instead, we will do a direct computation using \eqref{E:exf1} for type $D_r$.

\begin{prop} \label{P:D-poly}
For $G=\Spin_{2r}$ with $\Delta$ of type $D_r$ and $S_1=\set{\alpha_1}$, we have 
$$\mca{P}_{G, S_1}(X) = \frac{(X-1)((2r-1)X + (r-1)(2r-3))}{ \val{W} } \cdot \prod_{a=1}^{r-2} (X + 2a-1) \in \Q[X].$$
\end{prop}
\begin{proof}
As mentioned, we will have a direct computation using the equality
$$\sigma_{S_1^\vee}=\Ind_{W((S_1^\vee)^*)}^W \mbm{1}_{W((S_1^\vee)^*)} - \mbm{1}_W,$$
which gives
$$\angb{\sigma_{S_1^\vee}}{\eta_{\msc{X}_n}} = \angb{\mbm{1}_{W((S_1^\vee)^*)}}{ \eta_{\msc{X}_n}|_{W((S_1^\vee)^*)} }_{W((S_1^\vee)^*)} - \angb{\mbm{1}_W}{\eta_{\msc{X}_n}}_W.$$
We have
$$\angb{\mbm{1}_{W((S_1^\vee)^*)}}{ \eta_{\msc{X}_n}|_{W((S_1^\vee)^*)} }_{W((S_1^\vee)^*)} = \frac{1}{\val{W((S_1^\vee)^*)}} \sum_{w\in W((S_1^\vee)^*)} \chi_{\msc{X}_n}(w).$$
Let $Y':=\Z[(S_1^\vee)^*] \subset Y$ be the coroot lattice of type $D_{r-1}$. We have $Y' \cap nY = nY'$ and we set 
$$\msc{X}'_n:=Y'/(Y' \cap nY).$$
Similar to the action of $W=W(\Delta^\vee)$ on $\msc{X}_n$, the group $W((S_1^\vee)^*)$ acts naturally on $\msc{X}'_n$ and gives a permutation representation $\eta_{\msc{X}'_n}$. Denote by $\chi_{\msc{X}'_n}$ the character of $\eta_{\msc{X}'_n}$. It follows from \cite[Proposition 3.9]{Som97} that for oasitic $n$ one has
$$\chi_{\msc{X}_n}(w) = n\cdot \chi_{\msc{X}'_n}(w)$$
for every $w\in W((S_1^\vee)^*)$.

We have from \eqref{E:2extr} that
$$\angb{\mbm{1}_W}{\eta_{\msc{X}_n}}_W = \frac{1}{\val{W}} \cdot \prod_{j=1}^r (n + m_j),$$
where $m_j$ are the exponents of $W$. Similarly,
$$\sum_{w\in W((S_1^\vee)^*)} \chi_{\msc{X}'_n}(w) = \prod_{j=1}^{r-1}(n+m_j'),$$
where $m_j'$ are the exponents of the Weyl group $W((S_1^\vee)^*)$. Now the formula for $\angb{\sigma_{S_1^\vee}}{ \eta_{\msc{X}_n} }_W$ follows easily from the above and a simplification.
\end{proof}

\begin{rmk}
The above approach of a direct computation for type $D_r$ is clearly applicable to type $B_r$ or $C_r$. For example, for $G=\Spin_{2r+1}$ or $\Sp_{2r}$ and $S_1= \set{\alpha_1}$, we have
$$\mca{P}_{G, S_1}(X) = \frac{2r-1}{\val{W}} \cdot (X-1) \prod_{a=1}^{r-1} (X + 2a-1) \in \Q[X],$$
which is clearly compatible with (in view of Lemma \ref{L:obs} (i)) and is a refinement of Proposition \ref{P:BC-poly} (i). In fact, it is also possible to compute directly for $S_2, S_3$ for $\Delta$ of type $B_r, C_r$. 

On the other hand, we note that for type $A_r$ and general $S_j$, the computation of $\mca{P}_{G, S_j}(n) \in \Q[n]$ using the two-sided cells consideration is much more efficient than using the formula in \eqref{E:exf1}.
\end{rmk}

\subsection{Type $G_2$}
For $\Delta$ of the exceptional type $G_2$, the Whittaker polynomial $\mca{P}_{G, S}(X), S \in \msc{P}(\Delta)$ follows directly from \cite[Page 355, Table 9]{Ga6}. We reproduce the result in Table \ref{T:G2} below for completeness.

\begin{table}[!htbp]  
\caption{$\mca{P}_{G,S}(X)$ for $S \subseteq \Delta$}
\label{T:G2}
\vskip 5pt
\renewcommand{\arraystretch}{1.4}
\begin{tabular}{|c|c|c|c|c|}
\hline
$S$ & $\emptyset$  &  $\set{\alpha_1} $ & $\set{\alpha_2}$  & $\Delta$  \\
\hline
\hline
$12\cdot \mca{P}_{G,S}(X)$ & $ (X+1)(X+5)$ & $5(X^2-1)$  & $5(X^2-1)$ & $(X-1)(X-5)$ \\ 
\hline
\end{tabular}
\end{table}

\subsection{A speculation} \label{SS:spec}
We expect that the cases discussed above when the polynomial $\mca{P}_{G,S}(X) \in \Q[X]$ splits over $\Q$ is exhaustive. More precisely, we expect the following:

\begin{specu}
Let $G$ be almost simple and simply-connected with root system $\Delta$. Let $S \in \msc{P}(\Delta)$. Then $\mca{P}_{G,S}(X) \in \Q[X]$ splits over $\Q$ if and only if $S \in \mca{P}(\Delta)_\flat \cup \mca{P}(\Delta)_\flat^*$.
\end{specu}

Indeed, in quite many cases, different two-sided cells contribute to different polynomials in $\Q[X]$ in view of the results of \cite{GNS99}. Thus, if $\mfr{C}^{\rm LR}(S)$ is big, then it is more probable to have the non-splitting property of $\mca{P}_{G,S}(X)$: this together with some other numerical computations are the raisons d'\^etre of the above speculation.



\end{document}